\documentclass[11pt,reqno]{amsart}

\usepackage[utf8]{inputenc}
\usepackage[english]{babel}
\usepackage{url}
\usepackage[square,sort,comma,numbers]{natbib}
\usepackage{graphicx}
\usepackage[export]{adjustbox}
\usepackage{float}
\usepackage{amsmath}
\usepackage[toc,page]{appendix}
\usepackage{dirtytalk}
\usepackage{enumerate}

\usepackage{longtable}

\usepackage{amssymb,latexsym}
\usepackage{graphicx}
\usepackage{url}		

\calclayout
\makeatletter
\newcommand{\monthyear}[1]{%
	\def\@monthyear{\uppercase{#1}}}
\newcommand{\volnumber}[1]{%
	\def\@volnumber{\uppercase{#1}}}
\AtBeginDocument{%
	\def\ps@plain{\ps@empty
		\def\@oddfoot{\@monthyear \hfil \thepage}%
		\def\@evenfoot{\thepage \hfil \@volnumber}}
	\def\ps@firstpage{\ps@plain}
	\def\ps@headings{\ps@empty
		\def\@evenhead{%
			\setTrue{runhead}%
			\def\thanks{\protect\thanks@warning}%
			 \uppercase{ }\hfil}%
		\def\@oddhead{%
			\setTrue{runhead}%
			\def\thanks{\protect\thanks@warning}%
			\hfill\uppercase{New expansions for $x^n \pm y^n$ in terms of quadratic forms}}%
		\let\@mkboth\markboth
		\def\@evenfoot{%
			\thepage \hfil \@volnumber}%
		\def\@oddfoot{%
			\@monthyear \hfil \thepage}%
	}%
	\footskip=25pt
	\pagestyle{headings}%
}
\makeatother

\theoremstyle{plain}
\numberwithin{equation}{section}

\usepackage{amssymb,amsfonts}
\usepackage{amsmath}
\usepackage[all,arc]{xy}
\usepackage{enumerate}
\usepackage{mathrsfs}
\usepackage{tkz-graph}
\GraphInit[vstyle = Shade]
\tikzset{LabelStyle/.style = { rectangle, rounded corners, draw,minimum width = 2em, fill = yellow!50,
		text = red, font =\bfseries},VertexStyle/.append style = { inner sep=5pt,
		font = \large\bfseries},EdgeStyle/.append style = {->, bend left} }
\newtheorem{thm}{Theorem}[section]
\newtheorem{cor}[thm]{Corollary}

\newtheorem{lem}[thm]{Lemma}

\theoremstyle{definition}
\newtheorem{defn}[thm]{Definition}

\newtheorem{exmp}[thm]{Example}

\newtheorem{notn}[thm]{Notation}

\theoremstyle{remark}

\begin{document}
	\monthyear{February 9, 2020}
	\volnumber{Volume, Number}
	\setcounter{page}{1}
	
	\title{New expansions for $x^n \pm y^n$ in terms of quadratic forms }
	\author{Moustafa Ibrahim \\ Department of Mathematics, College of Science, \\ University of Bahrain, Kingdom of Bahrain}
	
	\address{Department of Mathematics, College of Science, University of Bahrain, Kingdom of Bahrain}
	\email{mimohamed@uob.edu.bh}

	\begin{abstract} We prove new theorems for the polynomial expansions of $x^n \pm y^n$ in terms of the binary quadratic forms $\alpha x^2 + \beta xy + \alpha y^2 $ and $a x^2 + bxy + a y^2 $. The paper gives new arithmetic differential approach to compute the coefficients. Also, the paper gives generalization to well-known polynomial identity in the history of number theory. The paper highlights the emergence of a new class of polynomials that unify many well-known sequences including the Chebyshev polynomials of the first and second kind, Dickson polynomials of the first and second kind, Lucas and Fibonacci numbers, Mersenne numbers, Pell polynomials, Pell-Lucas polynomials, and Fermat numbers. Also, this paper highlights the emergence of the notions of trajectories and orbits of certain integers that passes through many well-known polynomials and sequences. The Lucas-Fibonacci trajectory, the Lucas-Pell trajectory, the Fibonacci-Pell trajectory, the Fibonacci-Lucas trajectory, the Chebyshev-Dickson trajectory of the first kind, the Chebyshev-Dickson trajectory of the second kind, and others are new trajectories included in this paper. Also, the Lucas orbit, Fibonacci orbit, Mersenne orbit, Lucas-Fibonacci orbit, Fermat orbit, and others are new orbits included in this paper.  
\end{abstract}
	\maketitle
		\section{Summary Of The Main Results Of The Paper}
\subsection*{The Main Theorems of The Paper}	In this paper, we give new arithmetic differential approach to compute the coefficients of the polynomial expansions of $x^n\pm y^n$ in terms of the binary quadratic forms $\alpha x^2 + \beta xy + \alpha y^2  $ and $a x^2 + bxy + a y^2  $,  $\beta a - \alpha b \neq 0  $. Regarding the polynomials $x^n+y^n$, we prove the following new polynomial expansions for any given variables $\alpha, \beta, a , b$, and for any natural number $n$, 
	
	\begin{equation}
	\label{show example}
\begin{aligned}
&(\beta a  -  \alpha b)^{\lfloor{\frac{n}{2}}\rfloor} \frac{x^n+y^n}{(x+y)^{\delta(n)}} = \\
\sum_{r=0}^{\lfloor{\frac{n}{2}}\rfloor} \frac{(-1)^r}{r!}
(\alpha x^2 + \beta xy +& \alpha y^2)^{\lfloor{\frac{n}{2}}\rfloor -r} (ax^2+bxy+ay^{2})^{r}
\Big(\alpha \frac{{\partial} }{\partial a} + \beta \frac{{\partial}}{\partial b}\Big)^{r} \Psi(a,b,n) 
\end{aligned}
\end{equation}
where  $\delta(n)=1$ for $n$ odd and $\delta(n)=0$ for $n$ even and  $\lfloor{\frac{n}{2}}\rfloor$ is the highest integer less than or equal $\frac{n}{2}$. Regarding the polynomials $x^n-y^n$, we prove and study the following new polynomials expansions 
	\begin{equation}
		\label{show example 2}
\begin{aligned}	 	
&(\beta a -\alpha b)^{\lfloor{\frac{n-1}{2}}\rfloor} \frac{x^n-y^n}{(x-y)(x+y)^{\delta(n-1)}} = \\
\sum_{r=0}^{\lfloor{\frac{n-1}{2}}\rfloor} \frac{(-1)^r}{r!}
(\alpha& x^2 + \beta xy + \alpha y^2)^{\lfloor{\frac{n-1}{2}}\rfloor -r} (ax^2+bxy+ay^{2})^{r}
\Big(\alpha \frac{{\partial} }{\partial a} + \beta \frac{{\partial}}{\partial b}\Big)^{r} \Phi(a,b,n)
\end{aligned}
\end{equation}	
	
\subsection*{Computing The Polynomials $\Psi_r(a,b,n)$ and  $\Phi_r(a,b,n)$} The polynomials $\Psi_r(a,b,n)$ and  $\Phi_r(a,b,n)$ enjoy fascinating arithmetical and differential properties and unify many well-known polynomials. This paper presents three different approaches to compute the polynomials $\Psi_r(a,b,n)$ and  $\Phi_r(a,b,n)$. We  prove in this paper the following explicit formulas
	\begin{equation*}
	\begin{aligned}	 	
\Psi(a,b,n) &= \frac{(2a-b)^{\lfloor{\frac{n}{2}}\rfloor}} {2^n} \left\{  \left( 1 + \sqrt{ \frac{b+2a}{b-2a}} \right)^n + \left( 1 - \sqrt{ \frac{b+2a}{b-2a}} \right)^n \right\}, \\ 
\Phi(a,b,n) &= \frac{(2a-b)^{\lfloor{\frac{n-1}{2}}\rfloor}} {2^n \sqrt{\frac{b+2a}{b-2a}}} \left\{  \left( 1 + \sqrt{ \frac{b+2a}{b-2a}} \right)^n - \left( 1 - \sqrt{ \frac{b+2a}{b-2a}} \right)^n  \right\}.
\end{aligned}	 	
\end{equation*} 

\subsection*{Example For \eqref{show example}} To show simple example for \eqref{show example}, we take $n=4$. Simple computations show that 
\begin{equation}
\begin{aligned}
 \frac{(-1)^0}{0!}
\Big(\alpha \frac{{\partial} }{\partial a} + \beta \frac{{\partial}}{\partial b}\Big)^{0}  \Psi(a,b,4) &= \Psi(a,b,4) = -2 a^2 + b^2 , \\
 \frac{(-1)^1}{1!}
\Big(\alpha \frac{{\partial} }{\partial a} + \beta \frac{{\partial}}{\partial b}\Big)^{1}  \Psi(a,b,4) &= 4a \alpha - 2b \beta , \\
 \frac{(-1)^2}{2!}
\Big(\alpha \frac{{\partial} }{\partial a} + \beta \frac{{\partial}}{\partial b}\Big)^{2}  \Psi(a,b,4) &= -2 \alpha^2 + \beta^2.
\end{aligned}
\end{equation}
Therefore we get the following polynomial identity which is a special case for the main results of the current paper
\begin{align}
\begin{aligned}
\label{example}
(\beta a - \alpha b)^2 (x^4 + y^4) &= (-2 a^2 + b^2)(\alpha x^2 +\beta xy + \alpha y^2)^2 \\
&+ (4a \alpha - 2b \beta)(\alpha x^2 +\beta xy + \alpha y^2)(a x^2 +b xy + a y^2) \\
&+ (-2 \alpha^2 + \beta^2)(a x^2 + b xy + a y^2)^2.
\end{aligned}
\end{align}
 This particular example gives natural generalization to one of the well-known polynomial identities in the history of number theory as we see in section \eqref{history}.
\subsection*{The Emergence of New Class of Polynomials}
For any natural number $n$, and for any  $r = 0, 1, 2, \dots, \lfloor{\frac{n}{2}}\rfloor$, as we see in the current paper, the following polynomials \[  \frac{(-1)^r}{r!}
\Big(\alpha \frac{{\partial} }{\partial a} + \beta \frac{{\partial}}{\partial b}\Big)^{r} \Psi_r(a,b,n) \quad , \quad   \frac{(-1)^r}{r!}
\Big(\alpha \frac{{\partial} }{\partial a} + \beta \frac{{\partial}}{\partial b}\Big)^{r} \Phi_r(a,b,n)  \]
enjoy unexpected new arithmetical and differential properties. Actually, it gives new class of polynomials. Therefore, we need to define the following new polynomials 
\begin{equation*}
\begin{aligned}
\Psi\left(\begin{array}{cc|c}
a & b & n \\ \alpha & \beta & r \end{array} \right) &=  \frac{(-1)^r}{r!}
\Big(\alpha \frac{{\partial} }{\partial a} + \beta \frac{{\partial}}{\partial b}\Big)^{r} \Psi_r(a,b,n), \\
\Phi\left(\begin{array}{cc|c}
a & b & n \\ \alpha & \beta & r \end{array} \right) &=  \frac{(-1)^r}{r!}
\Big(\alpha \frac{{\partial} }{\partial a} + \beta \frac{{\partial}}{\partial b}\Big)^{r} \Phi_r(a,b,n).
\end{aligned}
\end{equation*}
 Also, for any $a,b,\alpha,\beta, \eta, \xi, n$, $\beta a - \alpha b \neq 0,$ we prove the following new and unexpected identities
	\begin{equation}
\begin{aligned}
\sum_{r=0}^{\lfloor{\frac{n}{2}}\rfloor}  \Psi\left( \begin{array}{cc|r} a & b & n \\ \alpha & \beta & r \end{array} \right)  \xi^{\lfloor{\frac{n}{2}}\rfloor  - r}  \eta^r =\Psi(a \xi-\alpha \eta, b \xi - \beta \eta, n), \\
\sum_{r=0}^{\lfloor{\frac{n-1}{2}}\rfloor}  \Phi\left( \begin{array}{cc|r} a & b & n \\ \alpha & \beta & r \end{array} \right)  \xi^{\lfloor{\frac{n-1}{2}}\rfloor  - r}  \eta^r =\Phi(a \xi-\alpha \eta, b \xi - \beta \eta, n).
\end{aligned}
\end{equation}
\subsection*{Unification of Well-Known Sequences} The polynomials $\Psi_r(a,b,n)$ and  $\Phi_r(a,b,n)$ enjoy fascinating arithmetical and differential properties and unify many well-known polynomials and sequences including the Chebyshev polynomials of the first and second kind, Dickson polynomials of the first kind and second kind, Lucas numbers, Fibonacci numbers, Fermat numbers, Pell-Lucas polynomials, Pell numbers and others, see \cite{222}, \cite{333}. 
\subsection*{Links With Fibonacci and Lucas Sequences} We should notice that $\Psi$ and $\Phi$ polynomials are greatly linked with Lucas and Fibonacci sequences through variety of formulas including the following formulas, for any variables $\alpha, \beta$,  
\begin{equation*}
\begin{aligned}
\Psi(-1,-3,n) &= L(n)  &, \qquad  \Psi(1,-3,n) &=
\begin{cases}
F(n)  & \mbox{for $n$ odd } \\
L(n) & \mbox{for $n$ even } \end{cases},    \\
\Phi(-1,-3,n) &= F(n)  &, \qquad  \Phi(1,-3,n) &=
\begin{cases}
F(n)  & \mbox{for $n$ even } \\
L(n) & \mbox{for $n$ odd } \end{cases},
\end{aligned}
\end{equation*}
where $L(n)$ are the Lucas numbers defined by the recurrence relation $L(0)=2,$ $L(1)=1$ and $L(n+1)=L(n)+L(n-1)$ and $F(n)$ are the Fibonacci numbers defined by the recurrence relation $F(0)=0,$ $F(1)=1$ and $F(n+1)=F(n)+F(n-1)$. 
\subsection*{Links With Dickson Polynomials} The Dickson polynomial of the first kind of degree $n$  with parameter $\alpha$, $D_n(x,\alpha)$, and the Dickson polynomial of the second kind of degree $n$  with parameter $\alpha$, $E_{n}(x,\alpha)$, are defined by the following formulas, where $\delta(n)=1$ for $n$ odd and $\delta(n)=0$ for $n$ even,
\begin{equation*}
\begin{aligned}
D_n(x,\alpha) =\sum_{i=0}^{\left\lfloor \frac{n}{2} \right\rfloor}\frac{n}{n-i} \binom{n-i}{i} (-\alpha)^i x^{n-2i} \quad , \quad
E_{n}(x,\alpha) =\sum_{i=0}^{\left\lfloor \frac{n}{2} \right\rfloor}  \binom{n-i}{i} (-\alpha)^i x^{n-2i}.
\end{aligned}
\end{equation*}
We show in this paper that the following relations hold 
\begin{equation*}
\begin{aligned}
D_n(x,\alpha) = x^{\delta(n)} \Psi(\alpha,2\alpha-x^2,n) \quad , \quad 
E_{n}(x,\alpha) = x^{\delta(n)} \Phi(\alpha,2\alpha-x^2,n+1).
\end{aligned}
\end{equation*}
\subsection*{Links With Chebyshev Polynomials} The Chebyshev polynomial of the first kind of degree $n$, $T_n(x)$, and the Chebyshev polynomial of the second kind of degree $n$, $U_{n}(x)$, are defined by the following formulas
\begin{equation*}
\begin{aligned}
T_n(x)=\sum_{i=0}^{\left\lfloor \frac{n}{2} \right\rfloor}(-1)^i \frac{n}{n-i} \binom{n-i}{i} (2)^{n-2i-1} x^{n-2i} \quad , \quad
U_{n}(x)=\sum_{i=0}^{\left\lfloor \frac{n}{2} \right\rfloor}  \binom{n-i}{i} (-1)^i (2x)^{n-2i}.
\end{aligned}
\end{equation*}
We show in this paper that the following relations hold 
\begin{equation*}
\begin{aligned}
T_n(x) = \frac{x^{\delta(n)}}{2^{\delta(n+1)}} \Psi(1,2-4x^2,n) \quad , \quad 
U_{n}(x) =(2x)^{\delta(n)} \Phi(1,2-4x^2,n+1).
\end{aligned}
\end{equation*}
\subsection*{Links With Mersenne and Fermat Numbers} 
The selections $(2,-5)$ and $(-2,-5)$ for $(a,b)$ will connect the $\Psi-$Polynomial with Mersenne numbers. Actually, for each natural number $n$ we get the following desirable formulas
\begin{equation}
\label{(2,-5)-(-2,-5)}
\begin{aligned}
\Psi(-2,-5,n)   &= 2^{n} +(-1)^{n}  &  &, & \Phi(-2,-5,n) &= \frac{2^{n} - (-1)^n}{3},              \\
\Psi(2,-5,n)    &= \frac{2^{n} +1}{3^{\delta(n)}}     &  &  ,  &   \Phi(2,-5,n) &= \frac{2^{n} - 1}{3^{\delta(n-1)}}.    \\
\end{aligned}
\end{equation}
and hence for $p$ odd we get $ \Psi(-2,-5,p) = \Phi(2,-5,p) = 2^{p} -1 = M_{p},$
where $M_{p}$  are called Mersenne numbers. Mersenne primes $M_p$, for some prime $p$, are also noteworthy due to their connection with primality tests, and perfect numbers. For some useful properties for the Mersenne numbers, the readers should consult with \cite{Mersenne}. In number theory, a perfect number is a positive integer that is equal to the sum of its proper positive divisors, that is, the sum of its positive divisors excluding the number itself. It is unknown whether there is any odd perfect number. Recently, \cite{25} showed that odd perfect numbers are greater than $10^{1500}$. Also, we should observe that $\Psi(-2,-5,2^n) = \Psi(2,-5,2^n) = 2^{2^n} + 1 = F_{n},$
where $F_n$ is the Fermat number.
\subsection*{Links With Pell Numbers and Polynomials} 
The Pell numbers $P_n$ are defined by the recurrence relation $ P_0 =0$, $P_1=1$, $P_n=2P_{n-1} + P_{n-2} $. The following formula shows that the $\Phi$ polynomial is also natural generalization for Pell numbers $ P_{n} = 2^{\delta(n-1)} \Phi(-1,-6,n) $. Actually, the Pell numbers $P_n=P_n(1)$, where $P_n(x)$ is the Pell polynomial defined by the recurrence relation $P_0(x)=0,$ $P_1(x)=1$ and $P_{n+1}(x)=2xP_{n}(x)+P_{n-1}(x)$. Similarly, we can deduce the following desirable relation $
P_{n}(x) = (2x)^{\delta(n-1)} \Phi(-1,-2-4x^2,n) $. The Pell-Lucas numbers $Q_n$ are defined by the recurrence relation $ Q_0= 2$,  $ Q_1= 2$, $ Q_n=2Q_{n-1}+Q_{n-2} $. The following formula shows that $\Psi$ polynomial is a generalization for Pell-Lucas numbers $ Q_{n} = 2^{\delta(n)} \Psi(-1,-6,n) $.
Also, the Pell-Lucas numbers $Q_n=Q_n(1)$, where $Q_n(x)$ is the Pell-Lucas polynomial defined by the recurrence relation $Q_0(x)=2, Q_1(x)=2x, Q_{n+1}(x)=2xQ_{n}(x)+Q_{n-1}(x)$. Similarly, we can deduce the following relation $ Q_{n}(x) = (2x)^{\delta(n)} \Psi(-1,-2-4x^2,n) $.

\subsection*{Trajectories and Orbits Connect Well-Known Polynomials and Sequences}
  The notions of \say{trajectories}, specific sequence of polynomials, and \say{orbits}, which are special cases of trajectories, are naturally arise up in this paper. The $\Psi$ and $\Phi$ polynomials are not only fundamental for the future developments of the study of the polynomial expansions of $x^n+y^n$ and $x^n-y^n$ in terms of binary quadratic forms but also to discover new  trajectories and orbits \say{connecting} well-known polynomials and numbers and relate them with unsolved Diophantine problems. For example, the following  particular trajectory  has arithmetic interest.  
  \subsection*{Trajectory From Sums of Powers to Another and Open Questions} 
  For any natural number $n$, we show in this paper that we can establish a trajectory from the polynomials $\Psi_r(n)$ that \say{connects} the polynomials $\frac{x^n+y^n}{(x+y)^{\delta(n)}}$ with the polynomials $\frac{z^n+t^n}{(z+t)^{\delta(n)}}$. \\
  
  \makeatletter
  \define@key{cylindricalkeys}{angle}{\def\myangle{#1}}
  \define@key{cylindricalkeys}{radius}{\def\myradius{#1}}
  \define@key{cylindricalkeys}{z}{\def\myz{#1}}
  \tikzdeclarecoordinatesystem{cylindrical}%
  {%
  	\setkeys{cylindricalkeys}{#1}%
  	\pgfpointadd{\pgfpointxyz{0}{0}{\myz}}{\pgfpointpolarxy{\myangle}{\myradius}}
  }
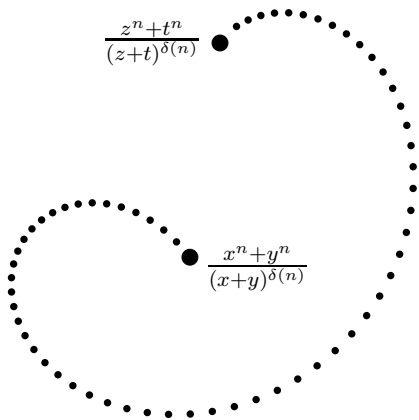
  \begin{tikzpicture}[z=0.2pt]
  \node (a) at (2.5,1.8) {$\frac{x^n+y^n}{(x+y)^{\delta(n)}}$};
  \node (b) at (1.1,4.8) {$\frac{z^n+t^n}{(z+t)^{\delta(n)}}$};
  \filldraw [] 
  (2,4.8) circle (3pt);
  \filldraw [] 
  (1.6,1.95) circle (3pt);
  \foreach \num in {60,67,...,490}
  \fill (cylindrical cs:angle=\num,radius=2,z=\num) circle (1.4pt);
  \node (c) at (4.2,2.2)  [right=0.7cm,text width=9cm,font=\footnotesize] 
  {We consider the following trajectory of polynomials 
  	\[ \frac{x^n+y^n}{(x+y)^{\delta(n)}} = \Psi_0(n), \Psi_1(n),\dots, \Psi_{\lfloor{\frac{n}{2}}\rfloor}(n) = \frac{z^n+t^n}{(z+t)^{\delta(n)}}  \]
  	where $\Psi_r(n) = \Psi\left( \begin{array}{cc|r} +xy & -x^2-y^2 & n \\ -zt & +z^2+t^2 & r \end{array} \right)$, which are the coefficients of the following polynomial expansion 
  	\[(zx-ty)^{\lfloor{\frac{n}{2}}\rfloor}(zy-tx)^{\lfloor{\frac{n}{2}}\rfloor}\frac{u^n+v^n}{(u+v)^{\delta(n)}} = \] 
  	\[  \sum_{r=0}^{\lfloor{\frac{n}{2}}\rfloor}
  	\Psi_r(n) (zu-tv)^{\lfloor{\frac{n}{2}}\rfloor -r} (zv-tu)^{\lfloor{\frac{n}{2}}\rfloor -r} ((ux-vy)^{r}(uy-vx))^{r} \] 
  };
  \end{tikzpicture}
  
  For this particular trajectory, we show that for any numbers $x,y,z,t,xz \neq yt, xt\neq yz$, any natural number $n$, the following identity hold 
  \[  \sum_{r=0}^{\lfloor{\frac{n}{2}}\rfloor} \Psi\left( \begin{array}{cc|r} +xy & -x^2-y^2 & n \\ -zt & +z^2+t^2 & r \end{array} \right) = \Psi(xy+zt,-x^2-y^2-z^2-t^2,n).
  \] 		
  This particular trajectory should be surrounded by lots of mathematical questions, from number theory to dynamical systems. One of these, when the endpoints coincident with each other, it gives an orbit. If such orbit exists, we get a solution to many unsolved Diophantine equations. We should ask, if such orbit exist, what are the general arithmetic characteristics of such orbit? For example, what is the characteristics of this orbit, if exist, for $n=5$. And it is natural to ask what is the formula for $n$ such that this orbit exist. Can such orbits exist for infinite values for $n$? Alternatively, this is equivalent to ask whether there is  any nontrivial integer solutions $(x,y,z,t)$ and any natural number $n$ for following Diophantine equation 
  \[     \frac{x^n+y^n}{(x+y)^{\delta(n)}} = \frac{z^n+t^n}{(z+t)^{\delta(n)}}. \]
  
      \subsection*{Fermat Orbit And Open Questions} We present in this paper what we call \say{Fermat orbit} which is the sequence $\eta(r)$ of integers
  \[ 2^{2^n}+1 = \eta(0), \eta(1), \eta(2), \dots, \eta(2^{n-1}) =  2^{2^n}+1,  \]
  which start with Fermat numbers $2^{2^n}+1$ and ended with Fermat numbers $2^{2^n}+1$ which are the coefficients of the following polynomial expansion 
  \[20^{{2^{n-1}}}( x^{2^n}+y^{2^n}) = \sum_{r=0}^{2^{n-1}}
 \eta(r) (-2 x^2 + 5 xy -2 y^2)^{2^{n-1} -r} (-2x^2-5xy-2y^{2})^{r}.\] 
  We show in this paper that $\eta(r) = \Psi\left( \begin{array}{cc|r} -2 & -5 & 2^n \\ -2 & +5 & r \end{array} \right)$.  
  
  \makeatletter
  \define@key{cylindricalkeys}{angle}{\def\myangle{#1}}
  \define@key{cylindricalkeys}{radius}{\def\myradius{#1}}
  \define@key{cylindricalkeys}{z}{\def\myz{#1}}
  \tikzdeclarecoordinatesystem{cylindrical}%
  {%
  	\setkeys{cylindricalkeys}{#1}%
  	\pgfpointadd{\pgfpointxyz{0}{0}{\myz}}{\pgfpointpolarxy{\myangle}{\myradius}}
  }
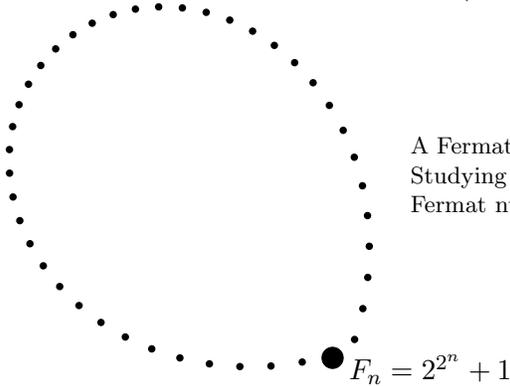
  \begin{tikzpicture}[z=0.2pt]
  \node (a) at (8.9,3.9) {$F_n=2^{2^n}+1 $};
  \filldraw [] 
  (7.6,4.05) circle (4pt);
  \foreach \num in {707,714,...,999}
  \fill (cylindrical cs:angle=\num,radius=3,z=\num) circle (1.4pt);
  \node (b) at (8,6.5)  [right=0.5cm,text width=9.5cm,font=\footnotesize] 
  {A Fermat prime is a Fermat number $F_n=2^{2^n}+1$ that is prime. Studying Fermat orbit should help understand the arithmetic of Fermat numbers which still are extremely ambiguous.};
  \end{tikzpicture} \\
   There are only five known Fermat primes \cite{Luca}. The five known Fermat primes are $F_0 = 3,   F_1 = 5,   F_2 = 17,   F_3 = 257,   F_4 = 65537$. Only seven Fermat numbers have been completely factored; $F_{5}, F_{6}, F_{7}, F_{8},F_{9},F_{10},F_{11}$. We should ask what is the arithmetic relations bewteen the prime factors of $\eta(1), \eta(2), \dots, \eta(2^{n-1})-1$, and $2^{2^n}+1$.
  Does an efficient algorithms exist based on the terms of Fermat orbit to factor $2^{2^n}+1$? This orbit should have links with the arithmetic of Fermat numbers.

  \section{DEFINITION OF $\Psi(a,b,n)$ and $\Phi(a,b,n)$ POLYNOMIALS}
  In \cite{1}, I first introduced  the definitions for $\Psi$ and $\Phi$ polynomials as following
  \begin{defn}
  	\label{Definition}
  	We define  $\delta(n)=1$ for $n$ odd and $\delta(n)=0$ for $n$ even. For any given variables $a,b$ and for any natural number $n$, we define the sequence
  	\[ \Psi(a,b,n)= \Psi(n),\quad \Phi(a,b,n)= \Phi(n), \]
  	by the following recurrence relations
  	\begin{equation}
  	\begin{aligned}
  	\label{def1}
  	\Psi(0)=2, \Psi(1)=1,\Psi(n+1)=(2a-b)^{\delta(n)}\Psi(n) - a \Psi(n-1),
  	\end{aligned}
  	\end{equation}
  	\begin{equation}
  	\begin{aligned}
  	\label{def2}
  	\Phi(0)=0, \Phi(1)= 1,\Phi(n+1)=(2a-b)^{\delta(n+1)}\Phi(n) - a \Phi(n-1)
  	\end{aligned}
  	\end{equation}
  \end{defn}

	\section{POLYNOMIAL EXPANSIONS IN TERMS OF BINARY QUADRATIC FORMS }
	In \cite{1}, page 447, using the formal derivation, we proved the following identities
	\begin{subequations}
	
	\begin{equation}
	\label{expansion 1}
	x^{n} + y^{n} =\sum_{i=0}^{\left\lfloor \frac{n}{2} \right\rfloor}(-1)^{i} \frac{n}{n-i}  \binom{n-i}{i} (xy)^i (x+y)^{n-2i}
	\end{equation}
	\begin{equation}
	\label{expansion 2}
	\frac{x^n-y^n}{x-y} =\sum_{i=0}^{\left\lfloor \frac{n-1}{2} \right\rfloor} (-1)^{i} \binom{n-i-1}{i} (xy)^i (x+y)^{n-2i-1}
	\end{equation}
\end{subequations}
	where $\lfloor{\frac{m}{2}}\rfloor$ denotes the largest integer $\leq  \frac{m}{2}$.
	Already Kummer and others used  the identities \eqref{expansion 1}, \eqref{expansion 2}, and for more details you can consult with \cite{Bini}, \cite{Boutin}, \cite{Gould} \cite{kummer}, \cite{Mention}, \cite{Ribenboim}, and \cite{Vachette}.
	\begin{thm}
		\label{special AA}
		For any given variables $\alpha, \beta, a , b$, and for any natural number $n$, there exist polynomials in $a,b, \alpha, \beta$ with integer coefficients, that we call
		$ \Psi\left( \begin{array}{cc|r} a & b & n \\ \alpha & \beta & r \end{array} \right) ,   \Phi\left( \begin{array}{cc|r} a & b & n \\ \alpha & \beta & r \end{array} \right) $
		 that depend only on $\alpha, \beta, a , b,n,$ and $r$, and satisfy the following polynomial identities
		\begin{equation}
		\label{special AAA}
		(\beta a - \alpha b)^{\lfloor{\frac{n}{2}}\rfloor} \frac{x^n+y^n}{(x+y)^{\delta(n)}} = \sum_{r=0}^{\lfloor{\frac{n}{2}}\rfloor}
		\Psi\left( \begin{array}{cc|r} a & b & n \\ \alpha & \beta & r \end{array} \right)
		(\alpha x^2 + \beta xy + \alpha y^2)^{\lfloor{\frac{n}{2}}\rfloor -r} (ax^2+bxy+ay^{2})^{r},
		\end{equation}
		\begin{equation}
		\label{special BBB}
		\begin{aligned}
		(\beta a - \alpha b)^{\lfloor{\frac{n-1}{2}}\rfloor}& \frac{x^n-y^n}{(x-y)(x+y)^{\delta(n-1)}} = \\
		&\sum_{r=0}^{\lfloor{\frac{n-1}{2}}\rfloor}
		\Phi\left( \begin{array}{cc|r} a & b & n \\ \alpha & \beta & r \end{array} \right)
		(\alpha x^2 + \beta xy + \alpha y^2)^{\lfloor{\frac{n-1}{2}}\rfloor -r} (ax^2+bxy+ay^{2})^{r}.
				\end{aligned}
		\end{equation}
	\end{thm}
	\begin{proof}
		From ~\eqref{expansion 1}, ~\eqref{expansion 2}, we obtain
		\begin{equation}
		\label{expansion 3}
		\frac {x^{n} + y^{n}}{ (x+y)^{\delta(n)}} =\sum_{i=0}^{\left\lfloor \frac{n}{2} \right\rfloor}(-1)^{i} \frac{n}{n-i}  \binom{n-i}{i} (xy)^i (x^2+2xy+y^2)^{\left\lfloor \frac{n}{2} \right\rfloor - i}
		\end{equation}
		and
		\begin{equation}
		\label{expansion 4}
		\frac{x^n-y^n}{(x-y)(x+y)^{\delta(n-1)}} =\sum_{i=0}^{\left\lfloor \frac{n-1}{2} \right\rfloor} (-1)^{i} \binom{n-i-1}{i} (xy)^i (x^2+2xy+y^2)^{\left\lfloor \frac{n-1}{2} \right\rfloor - i}
		\end{equation}
		Now, it is clear that for any variables $a,b, \alpha, \beta$, the following identities are true
		\begin{equation}
		\label{expansion 5}
		\begin{aligned}
		(\beta a-\alpha b)(x^2+2xy+y^2) = (2a-b)(\alpha x^2 + \beta xy + \alpha y^2) + (\beta- 2\alpha)( a x^2 + b xy + a y^2   ), \\
		\end{aligned}
		\end{equation}
		\begin{equation}
		\label{expansion 55}
		\begin{aligned}
		(\beta a-\alpha b) xy = a (\alpha x^2 + \beta xy + \alpha y^2) + (-\alpha)( a x^2 + b xy + a y^2   )
		\end{aligned}
		\end{equation}
		Now multiplying ~\eqref{expansion 3} by $(\beta a-\alpha b)^{\lfloor{\frac{n}{2}}\rfloor}$ and ~\eqref{expansion 4} by $(\beta a-\alpha b)^{\lfloor{\frac{n-1}{2}}\rfloor}$, with using  ~\eqref{expansion 5},  ~\eqref{expansion 55}, we obtain the proof.
	\end{proof}

\section{ALGEBRAIC INDEPENDENCE OF BINARY QUADRATIC FORMS}
For the following special case, applying Jacobian criterion for the special case $f_1, f_2$, \cite{Perron},  which says that polynomials $f_{1}$ and $f_{2}$ are algebraically independent if the determinant of the Jacobian $J$ is non zero where
\[ J =
\left( \begin{array}{cc}
\frac{\partial f_{1}}{\partial x} & \frac{\partial f_{1}}{\partial y} \\
\frac{\partial f_{2}}{\partial x}& \frac{\partial f_{2}}{\partial y}  \end{array} \right)
.\]
So working out with the Jacobian for the polynomials $f_1 = \alpha x^2 + \beta xy + \alpha y^2$ and $f_2= a x^2 + b xy + a y^2$, we get
\[J =
\left( \begin{array}{cc}
\frac{\partial f_{1}}{\partial x} & \frac{\partial f_{1}}{\partial y} \\
\frac{\partial f_{2}}{\partial x}& \frac{\partial f_{2}}{\partial y}  \end{array} \right) =  \left( \begin{array}{cc}
    2 \alpha x + \beta y         &    \beta x + 2 \alpha y          \\
       2 a x + by      &    b x + 2 a y
 \end{array} \right) .   \]
So $ |J| = 2 (\beta a - \alpha b)(y^2 - x^2 )$. Therefore we get the following

\begin{lem}
		\label{SYMMETRIC2}
For any given numbers $a,b, \alpha, \beta$, $\beta a - \alpha b \neq 0, x \neq y$, the binary quadratic forms $\alpha x^2 + \beta xy + \alpha y^2$ and $a x^2 + b xy + a y^2$ are algebraically independent.
\end{lem}

	\section{COMPUTING $\Psi$ AND $\Phi$ POLYNOMIALS FOR $r=0$ }
	
	\begin{thm}
		\label{BASE}
		For any numbers $a, b, \alpha, \beta,  n $, $2a-b \neq 0$, $\beta a - \alpha b \neq 0$, we have
		\begin{subequations}
			\begin{equation}
			\label{BASE1}
\Psi\left( \begin{array}{cc|r}
a & b & n \\ \alpha & \beta & 0 \end{array} \right) =\Psi(a,b,n) = \frac{(2a-b)^{\lfloor{\frac{n}{2}}\rfloor}} {2^n} \left\{  \left( 1 + \sqrt{ \frac{b+2a}{b-2a}} \right)^n + \left( 1 - \sqrt{ \frac{b+2a}{b-2a}} \right)^n \right\},
			\end{equation}
			\begin{equation}
			\label{BASE2}
		\Phi\left( \begin{array}{cc|r} a & b & n \\ \alpha & \beta & 0 \end{array} \right) =\Phi(a,b,n) = \frac{(2a-b)^{\lfloor{\frac{n-1}{2}}\rfloor}} {2^n \sqrt{\frac{b+2a}{b-2a}}} \left\{  \left( 1 + \sqrt{ \frac{b+2a}{b-2a}} \right)^n - \left( 1 - \sqrt{ \frac{b+2a}{b-2a}} \right)^n  \right\}.
				\end{equation}
		\end{subequations}
	\end{thm}
	\begin{proof}
		First, we prove that $  \Psi\left( \begin{array}{cc|r} a & b & n \\ \alpha & \beta & 0 \end{array} \right) = \Psi(a,b,n)$. Let $a, b, \alpha, \beta,  n $ be given numbers such that $2a-b \neq 0$, $\beta a - \alpha b \neq 0$. Define
		\[ H(n):= \Psi\left( \begin{array}{cc|r}
		a & b & n \\ \alpha & \beta & 0 \end{array} \right)\]
		From \eqref{special AAA}, the following congruence relations are true for any nonnegative integer $m$,
		\begin{equation}
		\label{H00}
		(\beta a - \alpha b)^{\lfloor{\frac{m}{2}}\rfloor} \frac{x^m+y^m}{(x+y)^{\delta(m)}}\equiv
		H(m) (\alpha x^2 + \beta xy + \alpha y^2)^{\lfloor{\frac{m}{2}}\rfloor}  \pmod{(ax^2+bxy+ay^{2})}
		\end{equation}
		We should notice that $H(0)=2, H(1)=1$. Now, multiply \eqref{H00} by
		$(2 a -  b)^{\lfloor{\frac{m}{2}}\rfloor}$ and noting from \eqref{expansion 5} that
		\[  (2a-b)(\alpha x^2 + \beta xy + \alpha y^2)  \equiv  (\beta a-\alpha b)(x^2+2xy+y^2) \pmod {ax^2 + b xy + a y^2},   \]
		we obtain
		\begin{equation}
		\begin{aligned}
		\label{H000}
	(2 a -  b)^{\lfloor{\frac{m}{2}}\rfloor} & (\beta a - \alpha b)^{\lfloor{\frac{m}{2}}\rfloor} \frac{x^m+y^m}{(x+y)^{\delta(m)}} \equiv \\
		 &(\beta a - \alpha b)^{\lfloor{\frac{m}{2}}\rfloor}
		H(m) ( x^2 + 2 xy +  y^2)^{\lfloor{\frac{m}{2}}\rfloor}  \pmod{(ax^2+bxy+ay^{2})}
		\end{aligned}
		\end{equation}
		As \eqref{H000} true for any $x,y$, $2a-b \neq 0$, $\beta a - \alpha b \neq 0$, we can remove  $(\beta a - \alpha b)^{\lfloor{\frac{m}{2}}\rfloor}$ from both sides to obtain the following congruence relation
		\begin{equation}
		\label{H0001}
		(2 a -  b)^{\lfloor{\frac{m}{2}}\rfloor} \frac{x^m+y^m}{(x+y)^{\delta(m)}}\equiv
		H(m) ( x^2 + 2 xy +  y^2)^{\lfloor{\frac{m}{2}}\rfloor}  \pmod{(ax^2+bxy+ay^{2})}
		\end{equation}
		Multiply both sides by $(x+y)^{\delta(m)}$, and noting that $\delta(m)+2\lfloor{\frac{m}{2}}\rfloor=m$, we obtain
		\begin{equation}
		\label{H0002}
		(2 a -  b)^{\lfloor{\frac{m}{2}}\rfloor} (x^m+y^m) \equiv
		H(m) ( x+y)^{m}  \pmod{(ax^2+bxy+ay^{2})}
		\end{equation}	
		Therefore, for any nonnegative integer $n$, we obtain the following three congruence relations
		\begin{subequations}
			\begin{equation}
			\label{H1a}
			(2 a -  b)^{\lfloor{\frac{n+1}{2}}\rfloor} (x^{n+1}+y^{n+1}) \equiv
			H(n+1) ( x+y)^{n+1}  \pmod{(ax^2+bxy+ay^{2})}
			\end{equation}
			\begin{equation}
			\label{H2a}
			(2 a -  b)^{\lfloor{\frac{n}{2}}\rfloor} (x^{n}+y^{n})  \equiv
			H(n) ( x+y)^{n}  \pmod{(ax^2+bxy+ay^{2})}	
			\end{equation}
			\begin{equation}
			\label{H3a}
			(2 a -  b)^{\lfloor{\frac{n-1}{2}}\rfloor} (x^{n-1}+y^{n-1})\equiv
			H(n-1) ( x+y)^{n-1} \pmod{(ax^2+bxy+ay^{2})}
			\end{equation}
		\end{subequations}
		It is clear that
		\begin{equation}
		\label{H3b}
		(2 a -  b)xy \equiv
		a(x+y)^2 \pmod{(ax^2+bxy+ay^{2})}
		\end{equation}
		Now consider the following polynomial identity
		\begin{equation}
		\label{H4a}
		\begin{aligned}
			&(2 a -  b)^{\lfloor{\frac{n+1}{2}}\rfloor}     (x^{n+1}+y^{n+1}) \\
			 &= (2 a -  b)^{\delta(n)}    (x+y) (2 a -  b)^{\lfloor{\frac{n}{2}}\rfloor}   ( x^{n}+y^{n}) - (2 a -  b)xy(2 a -  b)^{\lfloor{\frac{n-1}{2}}\rfloor}   ( x^{n-1}+y^{n-1})
		\end{aligned}
		\end{equation}
		From \eqref{H1a} \eqref{H2a}, \eqref{H3a}, \eqref{H3b}, and \eqref{H4a} we obtain
		\begin{equation}
		\label{HK}
		H(n+1) ( x+y)^{n+1} \equiv (2 a -  b)^{\delta(n)} H(n) ( x+y)^{n+1} - a H(n-1) ( x+y)^{n+1}
		\pmod{(ax^2+bxy+ay^{2})}
		\end{equation}
		As $2a-b \neq 0$, we can choose $x$ and $y$ such that $ax^2+bxy +ay^2$ and $x+y$ are relatively primes, and hence we can remove $( x+y)^{n+1}$ from \eqref{HK} to get
			\begin{equation}
			\label{HK1}
			H(n+1) \equiv (2 a -  b)^{\delta(n)} H(n)  - a  H(n-1)
			\pmod{(ax^2+bxy+ay^2)}
			\end{equation}
			Therefore the congruence \eqref{HK1}  must turn into identity and hence we obtain
				\begin{equation}
			\label{HK2}
		H(0)=2, H(1)=1,	H(n+1) = (2 a -  b)^{\delta(n)} H(n) - a  H(n-1)
			\end{equation}
		From \eqref{HK2} and \eqref{def1} we get $H(n) = \Psi(a,b,n)$.
		Now put
		\begin{equation}
		\begin{aligned}
 x = x_0 = \frac{1}{2}\Big(1 + \sqrt{\frac{b+2a}{b-2a}} \Big),  \\
 y = y_0 = \frac{1}{2}\Big(1 - \sqrt{\frac{b+2a}{b-2a}} \Big).  \\
	\end{aligned}
	\end{equation}
	Clearly $ax_0^2+bx_0y_0 +ay_0^2=0$, and hence the modulus of the congruence \eqref{H2a} vanish for every natural number $n$ which means that
		\begin{equation}
	\label{H2aF}
	(2 a -  b)^{\lfloor{\frac{n}{2}}\rfloor} (x_0^{n}+y_0^{n})  =
	H(n) ( x_0 + y_0)^{n}.
	\end{equation}
	Then, from \eqref{H2aF}, and as $x_0+y_0=1$, $H(n)=\Psi(a,b,n)$, we get
		\[    \Psi(a,b,n) = (2 a -  b)^{\lfloor{\frac{n}{2}}\rfloor} (x_0^{n}+y_0^{n}).\]
	Therefore, we obtain the full proof of \eqref{BASE1}. By the similar argument we can prove \eqref{BASE2}.
		\end{proof}
	Simple calculations, together with the polynomial identities ~\eqref{expansion 1}, ~\eqref{expansion 2}, and ~\eqref{BASE1}, ~\eqref{BASE2}, we get the following desirable polynomial expansions
	\begin{thm}
		\label{special EE}
		For any numbers $a,b$ and any natural number $n$, the following formulas are true
		\begin{subequations}
			\begin{equation}
			\label{P1}
			\Psi(a,b,n) =\sum_{i=0}^{\left\lfloor \frac{n}{2} \right\rfloor}\frac{n}{n-i} \binom{n-i}{i} (-a)^i (2a-b)^{\left\lfloor \frac{n}{2} \right\rfloor - i},
			\end{equation}
			\begin{equation}
			\label{P2}
			\Phi(a,b,n) =\sum_{i=0}^{\left\lfloor \frac{n-1}{2} \right\rfloor} \binom{n-i-1}{i} (-a)^i (2a-b)^{\left\lfloor \frac{n-1}{2} \right\rfloor - i}
			\end{equation}
		\end{subequations}
	\end{thm}
	\section{COMPUTING $\Psi$ AND $\Phi$ POLYNOMIALS}	
		Also, we should notice that there is some symmetry in ~\eqref{special AAA}, ~\eqref{special BBB} and
		\[(\beta a - \alpha b)^{\lfloor{\frac{m}{2}}\rfloor} =
		(-1)^{\lfloor{\frac{m}{2}}\rfloor}
		(b \alpha - a \beta )^{\lfloor{\frac{m}{2}}\rfloor}.\]
		Therefore we obtain the following desirable formulas
		\begin{thm}
			\label{BASE3}
			For any numbers $a, b, \alpha, \beta,  n $,  we have
			\begin{subequations}
				\begin{equation}
				\label{BASE4}
				\Psi\left( \begin{array}{cc|c}
				a & b & n \\ \alpha & \beta &\lfloor{\frac{n}{2}}\rfloor \end{array} \right) = (-1)^{\lfloor{\frac{n}{2}}\rfloor} \Psi(\alpha,\beta,n),
				\end{equation}
				\begin{equation}
				\label{BASE5}
				\Phi\left( \begin{array}{cc|c}
				a & b & n \\ \alpha & \beta &\lfloor{\frac{n-1}{2}}\rfloor \end{array} \right) = (-1)^{\lfloor{\frac{n-1}{2}}\rfloor} \Phi(\alpha,\beta,n).
				\end{equation}
			\end{subequations}
		\end{thm}
\section{NOTATIONS }
	\begin{notn}
Throughout this paper, and for specific given variables $a, b, \alpha , \beta $, we write
\[ \Psi_r(n) = \Psi\left( \begin{array}{cc|r} a & b & n \\ \alpha & \beta & r \end{array} \right) \quad , \quad  \Phi_r(n) = \Phi\left( \begin{array}{cc|r} a & b & n \\ \alpha & \beta & r \end{array} \right)  \]
\end{notn}

\begin{notn}
	Suppose that $P(x_i) = Q(x_i)$ is any polynomial identity in the variables $x_1, x_2, \dots, x_n$ and suppose $a_1, a_2, \dots, a_n$ are any given parameters. If we differentiate $P(x_i)=Q(x_i)$ with respect to the differential operator defined by $a_1 \frac{{\partial} }{\partial x_1} +  a_2 \frac{{\partial} }{\partial x_2}   +  \dots + a_n \frac{{\partial}}{\partial x_n}$, then we would say that we apply the differential map $\widetilde{\qquad} $ that sends $x_1 \longrightarrow a_1, x_2 \longrightarrow a_2, \dots , x_n \longrightarrow a_n $ to $P(x_i)=Q(x_i)$. Alternatively, we say that we differentiate $P(x_i)=Q(x_i)$ with respect to the differential map $\widetilde{\qquad} $ that sends $x_1 \longrightarrow a_1, x_2 \longrightarrow a_2, \dots , x_n \longrightarrow a_n$ if we apply the operator $ a_1 \frac{{\partial} }{\partial x_1} +  a_2 \frac{{\partial} }{\partial x_2}   +  \dots + a_n \frac{{\partial}}{\partial x_n} $
	to $P(x_i)=Q(x_i)$.
\end{notn}
			\begin{notn}
					Suppose that $f(x_i)$ is any polynomial in the variables $x_1, x_2, \dots, x_n$ and suppose $a_1, a_2, \dots, a_n$ are any given parameters. We define  \[ \overset{r}{\widetilde{f(x_i)}} :=  \Big( a_1 \frac{{\partial} }{\partial x_1}  +  a_2 \frac{{\partial} }{\partial x_2}   +  \dots + a_n \frac{{\partial}}{\partial x_n}  \Big)^{r} f(x_i)  \]
and	$ \overset{0}{\widetilde{f(x_i)}} :=  f(x_i) $. Moreover, we say that $ f(x_i)  \longrightarrow   g(x_i) $ with respect to the differential map$  \widetilde{\qquad}$that sends $x_1 \longrightarrow a_1, x_2 \longrightarrow a_2, \dots , x_n \longrightarrow a_n $ if		
	\[ g(x_i)  =   a_1 \frac{{\partial} }{\partial x_1} f(x_i) +  a_2 \frac{{\partial} }{\partial x_2} f(x_i)  +  \dots + a_n \frac{{\partial}}{\partial x_n}  f(x_i).     \]
\end{notn}
	\begin{notn}
If we list parameters, letters, $x_1, x_2, \dots, x_n$, we should look at every parameter in this paper as a variable. And parameter $x_k$ is considered constant with respect to a particular map $  \widetilde{\qquad}$ only if $ x_k \longrightarrow 0$.
\end{notn}
\begin{exmp}
	\label{example 1}
By differentiating $ax^2+bxy+ay^2$  with respect to the differential map
\[ a \longrightarrow \alpha ,  b \longrightarrow \beta, \alpha \longrightarrow 0 ,  \beta \longrightarrow 0,x \longrightarrow 0, y \longrightarrow 0, \]
we obtain
\begin{equation*}
\begin{aligned}
	 \Big(  \alpha \frac{{\partial} }{\partial a} +  \beta \frac{{\partial} }{\partial b}   +  0  \frac{{\partial}}{\partial \alpha}  + &  0  \frac{{\partial}}{\partial \beta} + 0  \frac{{\partial}}{\partial x}  +  0  \frac{{\partial}}{\partial y} \Big)  (ax^2+bxy+ay^2) \\ &=
	 \Big(  \alpha \frac{{\partial} }{\partial a} +  \beta \frac{{\partial} }{\partial b} \Big) (ax^2+bxy+ay^2) \\
	 &= \alpha \frac{{\partial} }{\partial a}(ax^2+bxy+ay^2) +  \beta \frac{{\partial} }{\partial b} (ax^2+bxy+ay^2) \\
	 &= \alpha x^2 + \beta xy + \alpha y^2  	
\end{aligned}
\end{equation*}
Therefore we say that $a x^2 + b xy + a y^2   \longrightarrow \alpha x^2 + \beta xy + \alpha y^2 $
with respect to the differential map $a \longrightarrow \alpha ,  b \longrightarrow \beta, \alpha \longrightarrow 0 ,  \beta \longrightarrow 0,x \longrightarrow 0, y \longrightarrow 0$. Also, we can show that by differentiating $(\beta a - \alpha b)^{\lfloor{\frac{n}{2}}\rfloor}$ with respect to the differential map \[ a \longrightarrow \alpha ,  b \longrightarrow \beta, \alpha \longrightarrow 0 ,  \beta \longrightarrow 0,x \longrightarrow 0, y \longrightarrow 0,\] we get
\begin{equation*}
\begin{aligned}
\Big(  \alpha \frac{{\partial} }{\partial a} +  \beta \frac{{\partial} }{\partial b}   + 0  \frac{{\partial}}{\partial \alpha}  +  0  \frac{{\partial}}{\partial \beta} + & 0  \frac{{\partial}}{\partial x}  +  0  \frac{{\partial}}{\partial y} \Big)  (\beta a - \alpha b)^{\lfloor{\frac{n}{2}}\rfloor} \\ &=
\Big(  \alpha \frac{{\partial} }{\partial a} +  \beta \frac{{\partial} }{\partial b} \Big) (\beta a - \alpha b)^{\lfloor{\frac{n}{2}}\rfloor} \\
&=  \lfloor{\frac{n}{2}}\rfloor  (\beta a - \alpha b)^{\lfloor{\frac{n}{2}}\rfloor -1}
\Big( \alpha \frac{{\partial} }{\partial a} +  \beta \frac{{\partial} }{\partial b}        \Big)   (\beta a - \alpha b)                              \\
&=  \lfloor{\frac{n}{2}}\rfloor  (\beta a - \alpha b)^{\lfloor{\frac{n}{2}}\rfloor -1}
\Big( \alpha \frac{{\partial} }{\partial a}(\beta a - \alpha b)      +  \beta \frac{{\partial} }{\partial b} (\beta a - \alpha b)            \Big)    \\
&=  \lfloor{\frac{n}{2}}\rfloor  (\beta a - \alpha b)^{\lfloor{\frac{n}{2}}\rfloor -1}
\Big( \alpha \beta \frac{{\partial} }{\partial a} a    -  \beta \alpha \frac{{\partial} }{\partial b}  b \Big) =     0   \\
\end{aligned}
\end{equation*}
Therefore we say that $(\beta a - \alpha b)^{\lfloor{\frac{n}{2}}\rfloor} \longrightarrow 0 $
with respect to the differential map \[a \longrightarrow \alpha ,  b \longrightarrow \beta, \alpha \longrightarrow 0 ,  \beta \longrightarrow 0,x \longrightarrow 0, y \longrightarrow 0.\]
\end{exmp}

\begin{exmp}
	With respect to the following particular differential map
	\[ a \longrightarrow -a-b ,  b \longrightarrow -b-4a, \alpha \longrightarrow -\alpha - \beta ,  \beta \longrightarrow -\beta -4 \alpha,x \longrightarrow y, y \longrightarrow x, \]
we get
\begin{equation*}
\begin{aligned}
\Big(  (-a-b) \frac{{\partial} }{\partial a} + &(-b-4a ) \frac{{\partial} }{\partial b}   + (-\alpha - \beta) \frac{{\partial}}{\partial \alpha}  + (-\beta -4 \alpha) \frac{{\partial}}{\partial \beta} + y  \frac{{\partial}}{\partial x}  +  x \frac{{\partial}}{\partial y} \Big) (ax^2+bxy+ay^2) \\
&= \Big( (-a-b) \frac{{\partial} }{\partial a} + (-b-4a ) \frac{{\partial} }{\partial b}   +  y  \frac{{\partial}}{\partial x}  +  x \frac{{\partial}}{\partial y} \Big) (ax^2+bxy+ay^2) \\
&=  (-a-b) (x^2+y^2) + (-b-4a ) (xy)  +  y  (2ax+by) +  x (bx+2ay)  \\
&= - (ax^2+bxy+ay^2)  \\
\end{aligned}
\end{equation*}
Therefore, for this particular differential map, we get
\[ ax^2+bxy+ay^2 \longrightarrow -(ax^2+bxy+ay^2) \]
Also, with respect to this particular differential map, we can easily check that
\begin{align*}
 \alpha x^2+ \beta xy+  \alpha y^2    &  \longrightarrow  - \alpha x^2 - \beta xy -  \alpha y^2 \\
(\beta a - \alpha b)      & \longrightarrow -2(\beta a - \alpha b)                                 \\
(x+y)^{\delta(n)}           & \longrightarrow
    \delta(n)  (x+y)^{\delta(n)}                     \\
x^n+y^n & \longrightarrow n (x^{n-1}y +y^{n-1}x)                                  \\
n (x^{n-1}y +y^{n-1}x)  & \longrightarrow n (x^{n} +y^{n}) + n(n-1)(x^{n-2}y^2+y^{n-2}x^2)                           \\
\end{align*}
  This particular differential map is useful for studying the expansions of the polynomials of the form   \[            x^n y^m + x^m y^n                      \]
\end{exmp}
\section{DIFFERENTIAL APPROACHES FOR COMPUTING $\Psi$ AND $\Phi$ POLYNOMIALS}
	For $\beta a - \alpha b \neq 0$, define
	 \[  \Psi_r(n) := \Psi\left( \begin{array}{cc|r} a & b & n \\ \alpha & \beta & r \end{array} \right)  \quad  \mbox{and} \quad \Phi_r(n) := \Phi\left( \begin{array}{cc|r} a & b & n \\ \alpha & \beta & r \end{array} \right) \]
To compute $\Psi_r(n), \Phi_r(n)$, consider the differential map $\widetilde{\qquad}$ that sends
\[ a \longrightarrow \alpha ,  b \longrightarrow \beta, \alpha \longrightarrow 0 ,  \beta \longrightarrow 0,x \longrightarrow 0, y \longrightarrow 0. \]
We differentiate ~\eqref{special AAA} with respect to this particular differential map
and noting from Example \eqref{example 1} the following
  \begin{align}
 \begin{aligned}
(\beta a - \alpha b)^{\lfloor{\frac{n}{2}}\rfloor}  & \longrightarrow 0                        &   , &
 &   \frac{x^n+y^n}{(x+y)^{\delta(n)}} & \longrightarrow 0.
   \end{aligned}
 \end{align}
Therefore
\begin{equation}
 (\alpha \frac{{\partial} }{\partial a} +  \beta \frac{{\partial} }{\partial b}) (\beta a - \alpha b)^{\lfloor{\frac{n}{2}}\rfloor} \frac{x^n+y^n}{(x+y)^{\delta(n)}} =0.
\end{equation}

Now acting  $  (\alpha \frac{{\partial} }{\partial a} +  \beta \frac{{\partial} }{\partial b}) $ on  ~\eqref{special AAA}, we obtain
\begin{equation}
\label{equation}
\begin{aligned}
0=\Big(\alpha \frac{{\partial} }{\partial a} & +  \beta \frac{{\partial} }{\partial b} \Big) \sum_{r=0}^{\lfloor{\frac{n}{2}}\rfloor}
\Psi_r(n) (\alpha x^2 + \beta xy + \alpha y^2)^{\lfloor{\frac{n}{2}}\rfloor -r} (ax^2+bxy+ay^{2})^{r}  \\
&=  \sum_{r=0}^{\lfloor{\frac{n}{2}}\rfloor}
(\alpha x^2 + \beta xy + \alpha y^2)^{\lfloor{\frac{n}{2}}\rfloor -r} (ax^2+bxy+ay^{2})^{r}\Big(\alpha \frac{{\partial} }{\partial a} +  \beta \frac{{\partial} }{\partial b} \Big)\Psi_r(n)\\
&+   \sum_{r=0}^{\lfloor{\frac{n}{2}}\rfloor}
\Psi_r(n)
(\alpha x^2 + \beta xy + \alpha y^2)^{\lfloor{\frac{n}{2}}\rfloor -r}\Big(\alpha \frac{{\partial} }{\partial a} +  \beta \frac{{\partial} }{\partial b} \Big) (ax^2+bxy+ay^{2})^{r} \\
&+  \sum_{r=0}^{\lfloor{\frac{n}{2}}\rfloor}
\Psi_r(n)
(ax^2+bxy+ay^{2})^{r}\Big(\alpha \frac{{\partial} }{\partial a} +  \beta \frac{{\partial} }{\partial b} \Big) (\alpha x^2 + \beta xy + \alpha y^2)^{\lfloor{\frac{n}{2}}\rfloor -r}
\end{aligned}
\end{equation}
Consequently, from \eqref{equation}, we obtain the following desirable polynomial identities
\begin{equation}
\label{special AAAA}
0 = \sum_{r=0}^{\lfloor{\frac{n}{2}}\rfloor} \Big(\widetilde{ \Psi_r(n)} + (r+1)\Psi_{r+1}(n) \Big)  (\alpha x^2 + \beta xy + \alpha y^2)^{\lfloor{\frac{n}{2}}\rfloor -r} (ax^2+bxy+ay^{2})^{r}.
\end{equation}
where, according to our notations, $ \widetilde{ \Psi_r(n)} :=
(\alpha \frac{{\partial} }{\partial a} +  \beta \frac{{\partial} }{\partial b})\Psi_r(n)$. Similarly, we differentiate ~\eqref{special BBB} with same particular differential map. We should have no difficulty to obtain the following identity
\begin{equation}
\label{special BBBB}
0 = \sum_{r=0}^{\lfloor{\frac{n-1}{2}}\rfloor} \Big(\widetilde{ \Phi_r(n)} + (r+1)\Phi_{r+1}(n) \Big)  (\alpha x^2 + \beta xy + \alpha y^2)^{\lfloor{\frac{n-1}{2}}\rfloor -r} (ax^2+bxy+ay^{2})^{r}
\end{equation}
where, according to our notations, $ \widetilde{ \Phi_r(n)} :=
(\alpha \frac{{\partial} }{\partial a} +  \beta \frac{{\partial} }{\partial b})\Phi_r(n)$. As $ \beta a - \alpha b \neq 0 $, from Lemma \eqref{SYMMETRIC2}, the polynomials $(\alpha x^2 + \beta xy + \alpha y^2)$ and $(a x^2 + b xy + a y^2)$ are algebraic independent which means that all of the coefficients of ~\eqref{special AAAA}, and ~\eqref{special BBBB} must vanish. This means that with respect to the differential map $ a \longrightarrow \alpha ,  b \longrightarrow \beta, \alpha \longrightarrow 0 ,  \beta \longrightarrow 0$, we obtain
\begin{align}
\begin{aligned}
\widetilde{ \Psi_r(n)} + (r+1)\Psi_{r+1}(n) &= 0 , \quad \forall_{r=0}^{\lfloor{\frac{n}{2}}\rfloor - 1} r \\
\widetilde{ \Phi_r(n)} + (r+1)\Phi_{r+1}(n) &= 0 , \quad \forall_{r=0}^{\lfloor{\frac{n-1}{2}}\rfloor - 1} r
\end{aligned}
\end{align}
Similarly, with respect to the differential map $ \alpha \longrightarrow a ,  \beta \longrightarrow b, a \longrightarrow 0 ,  b \longrightarrow 0 $, we get
\begin{align}
\begin{aligned}
	\widetilde{ \Psi_r(n)} + (\lfloor{\frac{n}{2}}\rfloor - r + 1 )\Psi_{r-1}(n) &= 0 , \quad \forall_{r=1}^{\lfloor{\frac{n}{2}}\rfloor } r \\
		\widetilde{ \Phi_r(n)} +(\lfloor{\frac{n-1}{2}}\rfloor - r +1 )\Phi_{r-1}(n) &= 0 , \quad \forall_{r=1}^{\lfloor{\frac{n-1}{2}}\rfloor } r
\end{aligned}
\end{align}
This immediately gives the following desirable theorems
\begin{thm}
	\label{DDD}
		For any natural number $n$, for any parameters $a,b,\alpha,\beta, \beta a - \alpha b \neq 0$, let
		\[  \Psi_r(n) := \Psi\left( \begin{array}{cc|r} a & b & n \\ \alpha & \beta & r \end{array} \right)  \quad  \mbox{and} \quad \Phi_r(n) := \Phi\left( \begin{array}{cc|r} a & b & n \\ \alpha & \beta & r \end{array} \right). \] Consider the differential map $\widetilde{\qquad}$ that sends \[a \longrightarrow \alpha ,  b \longrightarrow \beta, \alpha \longrightarrow 0 ,  \beta \longrightarrow 0.\]
		Then with respect to this differential map we get
\begin{align}
\begin{aligned}
\Psi_{r}(n) = \frac{-1}{r} \widetilde{ \Psi_{r-1}(n)}  \qquad \forall_{r=1}^{\lfloor{\frac{n}{2}}\rfloor} r  \qquad ,  \qquad
\Phi_{r}(n) = \frac{-1}{r} \widetilde{ \Phi_{r-1}(n)}   \quad \forall_{r=1}^{\lfloor{\frac{n-1}{2}}\rfloor} r \\
\end{aligned}
\end{align}
\end{thm}
\begin{thm}
	\label{DDD Pluse}
	For any natural number $n$, for any parameters $a,b,\alpha,\beta, \beta a - \alpha b \neq 0$, let
	\[  \Psi_r(n) := \Psi\left( \begin{array}{cc|r} a & b & n \\ \alpha & \beta & r \end{array} \right)  \quad  \mbox{and} \quad \Phi_r(n) := \Phi\left( \begin{array}{cc|r} a & b & n \\ \alpha & \beta & r \end{array} \right). \]  Consider the differential map $\widetilde{\qquad}$ that sends \[\alpha \longrightarrow a ,  \beta \longrightarrow b, a \longrightarrow 0 ,  b \longrightarrow 0.\]
	Then with respect to this differential map we get
	\begin{align}
	\begin{aligned}
	\Psi_{r}(n) = \frac{-1}{\lfloor{\frac{n}{2}}\rfloor - r} \widetilde{ \Psi_{r+1}(n)}  \qquad \forall_{r=0}^{\lfloor{\frac{n}{2}}\rfloor -1} r  \qquad ,  \qquad
	\Phi_{r}(n) = \frac{-1}{\lfloor{\frac{n}{2}}\rfloor - r} \widetilde{ \Phi_{r+1}(n)}  \quad \forall_{r=0}^{\lfloor{\frac{n-1}{2}}\rfloor -1} r \\
	\end{aligned}
	\end{align}
\end{thm}
\subsection*{FIRST OBSERVATION ABOUT THEOREM \eqref{DDD}}
Theorem \eqref{DDD} has lots of applications. We should realize that Theorem \eqref{DDD} is also true for any differentiable map $\widetilde{\qquad}$ as long as $\widetilde{a} = \alpha, \quad \widetilde{b} = \beta, \quad \widetilde{\alpha} = 0, \quad \widetilde{\beta} = 0, \quad   \beta a - \alpha b \neq 0$. Therefore, we should easily get the following desirable generalization
\begin{cor}
	\label{VVV}
	For any natural number $n$, for any parameters $a,b,\alpha,\beta, \theta, \beta a - \alpha b \neq 0$, let
\[  \Psi_r(n) := \Psi\left( \begin{array}{cc|r} a-\alpha \theta & b - \beta \theta & n \\ \alpha & \beta & r \end{array} \right)  \quad  \mbox{and} \quad \Phi_r(n) := \Phi\left( \begin{array}{cc|r} a-\alpha \theta & b - \beta \theta & n \\ \alpha & \beta & r \end{array} \right). \] Consider the differential map $\widetilde{\qquad}$ that sends $ a \longrightarrow 0 ,  b \longrightarrow 0, \alpha \longrightarrow 0 ,  \beta  \longrightarrow 0, \theta \longrightarrow -1$. Then with respect to this particular differential map we get
\begin{align}
\begin{aligned}
\Psi_{r}(n) &= \frac{-1}{r} \widetilde{ \Psi_{r-1}(n)}   , \quad \forall_{r=1}^{\lfloor{\frac{n}{2}}\rfloor} r   \\
\Phi_{r}(n) &= \frac{-1}{r} \widetilde{ \Phi_{r-1}(n)}   , \quad \forall_{r=1}^{\lfloor{\frac{n-1}{2}}\rfloor} r.
\end{aligned}
\end{align}
\end{cor}
\subsection*{SECOND OBSERVATION ABOUT THEOREM \eqref{DDD}}
In terms of the classical notations of the partial differentiation we get
\begin{thm}
		\label{TT2}
	The polynomials $\Psi\left(\begin{array}{cc|c}
	a & b & n \\ \alpha & \beta & r \end{array}\right)$ and $\Phi\left(\begin{array}{cc|c}
	a & b & n \\ \alpha & \beta & r \end{array}\right)$   satisfy
	\begin{align}
	\begin{aligned}
	\Psi\left(\begin{array}{cc|c}
	a & b & n \\ \alpha & \beta & r \end{array} \right) &=  \frac{(-1)^r}{r!}
	\Big(\alpha \frac{{\partial} }{\partial a} + \beta \frac{{\partial}}{\partial b}\Big)^{r} \Psi(a,b,n),  \\
	\Psi\left( \begin{array}{cc|c}
	a & b & n \\ \alpha & \beta & r \end{array} \right) &=  \frac{(-1)^r}{(\lfloor{\frac{n}{2}}\rfloor -r)!} \Big(a \frac{{\partial} }{\partial \alpha} + b \frac{{\partial}}{\partial \beta}\Big)^{\lfloor{\frac{n}{2}}\rfloor-r} \Psi(\alpha,\beta,n),  \\
		\Phi\left( \begin{array}{cc|c}
	a & b & n \\ \alpha & \beta & r \end{array} \right) &=  \frac{(-1)^r}{r!}
	\Big(\alpha \frac{{\partial} }{\partial a} + \beta \frac{{\partial}}{\partial b}\Big)^{r} \Phi(a,b,n),  \\
	\Phi\left( \begin{array}{cc|c}
	a & b & n \\ \alpha & \beta & r \end{array} \right) &=  \frac{(-1)^r}{(\lfloor{\frac{n-1}{2}}\rfloor -r)!} \Big(a \frac{{\partial} }{\partial \alpha} + b \frac{{\partial}}{\partial \beta}\Big)^{\lfloor{\frac{n-1}{2}}\rfloor-r} \Phi(\alpha,\beta,n).
	\end{aligned}
	\end{align}
\end{thm}
An immediate consequence is the following result
\begin{thm}
	\label{orbits}
	For any natural numbers $n$, we have
	\begin{align}
	\begin{aligned}
	\frac{1}{(\lfloor{\frac{n}{2}}\rfloor)!} \Big(\alpha \frac{{\partial} }{\partial a} + \beta \frac{{\partial}}{\partial b}\Big)^{\lfloor{\frac{n}{2}}\rfloor} \Psi(a,b,n) &= \Psi(\alpha,\beta,n)  \\
	\frac{1}{(\lfloor{\frac{n-1}{2}}\rfloor)!} \Big(\alpha \frac{{\partial} }{\partial a} + \beta \frac{{\partial}}{\partial b}\Big)^{\lfloor{\frac{n-1}{2}}\rfloor} \Phi(a,b,n) &= \Phi(\alpha,\beta,n)
	\end{aligned}
	\end{align}
\end{thm}
\subsection*{ILLUSTRATIVE EXAMPLE TO APPLY THE METHODS OF THEOREM \eqref{DDD}} It is desirable to clarify how we apply the methods of Theorem\eqref{DDD} to compute the polynomial coefficients $ \Psi\left( \begin{array}{cc|r} a & b & n \\ \alpha & \beta & r \end{array} \right)$ or $\Phi\left( \begin{array}{cc|r} a & b & n \\ \alpha & \beta & r \end{array} \right)$, for any specific value for $n$. Therefore, in the next section, we choose to compute
\[\Psi\left( \begin{array}{cc|r} a & b & 4 \\ \alpha & \beta & r \end{array} \right) \quad \mbox{for} \quad r=0,1,2=\left\lfloor \frac{4}{2} \right\rfloor. \]

\subsection*{LINKS WITH WELL-KNOWN IDENTITY IN THE HISTORY OF NUMBER THEORY }
\label{history}
For example, take $n=4$. Then \[
\Psi\left( \begin{array}{cc|c}
a & b & 4 \\ \alpha & \beta & 0 \end{array} \right) =\Psi(a,b,4)= -2a^2+b^2.   \]
Hence
\begin{align*}
\begin{aligned}
\Psi\left( \begin{array}{cc|c}
a & b & 4 \\ \alpha & \beta & 1 \end{array} \right)  &=   \frac{-1}{1}\Big(\alpha \frac{{\partial} }{\partial a} +  \beta \frac{{\partial} }{\partial b} \Big)(-2a^2+b^2)= 4a \alpha - 2b \beta , \\
\Psi\left( \begin{array}{cc|c}
a & b & 4 \\ \alpha & \beta & 2 \end{array} \right)  &= \frac{-1}{2}  \Big(\alpha \frac{{\partial} }{\partial a} + \beta \frac{{\partial}}{\partial b} \Big) (4a \alpha - 2b \beta) = - 2 \alpha^2 + \beta^2. \\
\end{aligned}
\end{align*}
Then from ~\eqref{special AAA} we immediately obtain the following polynomial identity
\begin{align}
\begin{aligned}
\label{generalization-n-4}
(\beta a - \alpha b)^2 (x^4 + y^4) &= (-2 a^2 + b^2)(\alpha x^2 +\beta xy + \alpha y^2)^2 \\
&+ (4a \alpha - 2b \beta)(\alpha x^2 +\beta xy + \alpha y^2)(a x^2 +b xy + a y^2) \\
&+ (-2 \alpha^2 + \beta^2)(a x^2 + b xy + a y^2)^2
\end{aligned}
\end{align}
Generally, it is desirable to select values for the parameters $a, b, \alpha, \beta$ to make the middle term, $ \Psi\left( \begin{array}{cc|c} a & b & 4 \\ \alpha & \beta & 1 \end{array} \right)  $, equal zero. Therefore, put $ \left( \begin{array}{cc} a & b \\ \alpha & \beta \end{array} \right) = \left( \begin{array}{cc} 1 & 1 \\ 1 & 2 \end{array} \right) $. Then the middle coefficient $4a \alpha - 2b \beta$ of ~\eqref{generalization-n-4}  vanish, and we obtain the following special case for a well-known identity in the history of number theory that is used extensively in the  study of equal sums of like powers and in discovering new formulas for Fibonacci numbers
\begin{align}
\begin{aligned}
x^4 + y^4 + (x+y)^4 = 2 (x^2 + xy + y^2)^2
\end{aligned}
\end{align}
In volume 2, on page $650$, \cite{Dickson} attributes this special case to C. B. Haldeman (1905), although Proth (1878) used it in passing (see page $657$ of \cite{Dickson}).

	\section{COMMON FACTORS}
Also, it is useful to note the following relations. Replace each $\alpha$ and $\beta$ by $\lambda \alpha$ and $\lambda \beta$ respectively in \eqref{special AAA}, we get the following polynomial identity for any $\lambda$
	\begin{equation*}
\label{}
\begin{aligned}
&(\lambda\beta a - \lambda\alpha b)^{\lfloor{\frac{n}{2}}\rfloor} \frac{x^n+y^n}{(x+y)^{\delta(n)}} = \\
&\sum_{r=0}^{\lfloor{\frac{n}{2}}\rfloor}
\Psi\left( \begin{array}{cc|r} a & b & n \\ \lambda\alpha & \lambda\beta & r \end{array} \right)
(\lambda\alpha x^2 + \lambda\beta xy + \lambda \alpha y^2)^{\lfloor{\frac{n}{2}}\rfloor -r} (ax^2+bxy+ay^{2})^{r}
\end{aligned}
\end{equation*}
Then
	\begin{equation}
\label{Vcommon}
\begin{aligned}
 &(\beta a - \alpha b)^{\lfloor{\frac{n}{2}}\rfloor} \frac{x^n+y^n}{(x+y)^{\delta(n)}} = \\ &\sum_{r=0}^{\lfloor{\frac{n}{2}}\rfloor} \lambda^{ -r}
\Psi\left( \begin{array}{cc|r} a & b & n \\ \lambda\alpha & \lambda\beta & r \end{array} \right)
(\alpha x^2 + \beta xy +  \alpha y^2)^{\lfloor{\frac{n}{2}}\rfloor -r} (ax^2+bxy+ay^{2})^{r}
\end{aligned}
\end{equation}
Comparing \eqref{Vcommon} with \eqref{special AAA}, we obtain
\[     \lambda^{ -r}
\Psi\left( \begin{array}{cc|r} a & b & n \\ \lambda\alpha & \lambda\beta & r \end{array} \right) =     \Psi\left( \begin{array}{cc|r} a & b & n \\ \alpha & \beta & r \end{array} \right)       \]
Similarly, we can prove the following useful relations.

\begin{thm}
	\label{Wcommon}
	For any numbers $a, b, \alpha, \beta, \beta a - \alpha b \neq 0, \lambda, r,  n $,  we get
	\begin{subequations}
		\begin{equation}
		\label{Wcommon1}
	   \Psi\left( \begin{array}{cc|c}
	   a & b & n \\ \lambda \alpha & \lambda \beta & r \end{array} \right)
	   =\lambda^{r} \Psi\left( \begin{array}{cc|c}
	   a & b & n \\ \alpha & \beta & r \end{array} \right),
		\end{equation}
		\begin{equation}
		\label{Wcommon2}
	          \Phi\left( \begin{array}{cc|c}
	a & b & n \\ \lambda \alpha & \lambda \beta & r \end{array} \right)
	=\lambda^{r} \Phi\left( \begin{array}{cc|c}
	a & b & n \\ \alpha & \beta & r \end{array} \right),
		\end{equation}
		\begin{equation}
		\label{Wcommon3}
		\Psi\left( \begin{array}{cc|c}
		\lambda a &\lambda b & n \\ \alpha & \beta & r \end{array} \right)
		=\lambda^{\lfloor{\frac{n}{2}}\rfloor - r}
		\Psi\left( \begin{array}{cc|c}
		a & b & n \\ \alpha & \beta & r \end{array} \right),
		\end{equation}
		\begin{equation}
		\label{Wcommon4}
		   \Phi\left( \begin{array}{cc|c}
			\lambda a &\lambda b & n \\ \alpha & \beta & r \end{array} \right)
	=\lambda^{\lfloor{\frac{n-1}{2}}\rfloor - r}
		\Phi\left( \begin{array}{cc|c}
			a    &   b   & n \\
			\alpha & \beta & r 	\end{array} \right),
				\end{equation}
				\begin{equation}
		\label{}
		\Psi\left( \begin{array}{cc|c}
		a & b & n \\ \alpha & \beta & r \end{array} \right) =
		(-1)^{\lfloor{\frac{n}{2}} \rfloor}
		\Psi\left( \begin{array}{cc|c}
		\alpha & \beta  & n \\  a & b &\lfloor{\frac{n}{2}} \rfloor - r \end{array} \right),
				\end{equation}	
				\begin{equation}
		\label{}
	\Phi\left( \begin{array}{cc|c}
	a & b & n \\ \alpha & \beta & r \end{array} \right) =
	(-1)^{\lfloor{\frac{n-1}{2}} \rfloor}
	\Phi\left( \begin{array}{cc|c}
	\alpha & \beta  & n \\  a & b &\lfloor{\frac{n-1}{2}} \rfloor - r \end{array} \right),	
		\end{equation}
			\begin{equation}
		\label{common factor1}
		\lambda^{\lfloor{\frac{n}{2}}\rfloor}\Psi(a,b,n) = \Psi(\lambda a,\lambda  b,n),   		
		\end{equation}
		\begin{equation}
		\label{common factor2}
		\lambda^{\lfloor{\frac{n-1}{2}}\rfloor}\Phi(a,b,n)  =  \Phi(\lambda a,\lambda  b,n).
		\end{equation}		
	\end{subequations}
\end{thm}

\section{FORMULAS FOR SUMS INCLUDING $\Psi$ AND $\Phi$ POLYNOMIALS }
We now prove the following theorem.
\begin{thm}
	\label{3R3}
	For any $a,b,\alpha,\beta, \theta, n$, $\beta a - \alpha b \neq 0,$ the following identities are true
	
	\begin{subequations}
		\begin{equation}
	\label{R3}
	\sum_{r=0}^{\lfloor{\frac{n}{2}}\rfloor}  \Psi\left( \begin{array}{cc|r} a & b & n \\ \alpha & \beta & r \end{array} \right) \theta^r = \Psi(a-\alpha \theta , b - \beta \theta, n),
			\end{equation}
		\begin{equation}
	\label{R3R}
	\sum_{r=0}^{\lfloor{\frac{n-1}{2}}\rfloor}  \Phi\left( \begin{array}{cc|r} a & b & n \\ \alpha & \beta & r \end{array} \right) \theta^r = \Phi(a-\alpha \theta , b - \beta \theta, n)
			\end{equation}
				
		\end{subequations}
\end{thm}
\begin{proof}
	Define $q_1:=\alpha x^2+\beta xy +\alpha y^2$ and $q_2:=a x^2+b xy +a y^2$ and
	\[ \Lambda_\theta :=\theta q_1 - q_2 = (\alpha \theta - a)x^2 + (\beta \theta - b)xy + (\alpha \theta - a)y^2 .\] 	
	From \eqref{special AAA}, we know that
	\[
	(\beta a - \alpha b)^{\lfloor{\frac{n}{2}}\rfloor} \frac{x^n+y^n}{(x+y)^{\delta(n)}} = \sum_{r=0}^{\lfloor{\frac{n}{2}}\rfloor}  \Psi\left( \begin{array}{cc|r} a & b & n \\ \alpha & \beta & r \end{array} \right) (q_1)^{\lfloor{\frac{n}{2}}\rfloor -r} (q_2)^{r}, \]
	As $q_2 \equiv \theta q_1  \pmod{\Lambda_\theta}$, we get
	\begin{equation}
	\label{R4}
	(\beta a - \alpha b)^{\lfloor{\frac{n}{2}}\rfloor} \frac{x^n+y^n}{(x+y)^{\delta(n)}} \equiv q_1^{\lfloor{\frac{n}{2}}\rfloor} \sum_{r=0}^{\lfloor{\frac{n}{2}}\rfloor}  \Psi\left( \begin{array}{cc|r} a & b & n \\ \alpha & \beta & r \end{array} \right) \theta^r   \pmod{\Lambda_\theta}
	\end{equation}
		Replace each of $a,b$ by $\alpha \theta - a, \beta \theta - b$ respectively,  in \eqref{special AAA}, we obtain
	\begin{equation}
	\label{}
	(\beta [\alpha \theta - a] -
	\alpha[\beta \theta - b])^{\lfloor{\frac{n}{2}}\rfloor} \frac{x^n+y^n}{(x+y)^{\delta(n)}} \equiv
	\Psi(\alpha \theta - a, \beta \theta - b, n)  q_1^{\lfloor{\frac{n}{2}}\rfloor} \pmod{\Lambda_\theta}
	\end{equation}	
	As $\beta [\alpha \theta - a] -
	\alpha[\beta \theta - b] = - (\beta a - \alpha b)$, and noting from \eqref{common factor1} that
	\[ (-1)^{\lfloor{\frac{n}{2}}\rfloor} \Psi(\alpha \theta - a, \beta \theta - b, n) =   \Psi(a-\alpha \theta , b - \beta \theta, n), \]
	we obtain the following congruence
		\begin{equation}
	\label{two-I}
	(  \beta a - \alpha b  )^{\lfloor{\frac{n}{2}}\rfloor} \frac{x^n+y^n}{(x+y)^{\delta(n)}} \equiv
	\Psi(a-\alpha \theta , b - \beta \theta, n) q_1^{\lfloor{\frac{n}{2}}\rfloor} \pmod{\Lambda_\theta}
	\end{equation}
	Now, subtracting \eqref{R4} and\eqref{two-I}, we obtain
		\begin{equation}
	\label{two-II}
	0 \equiv \Big(
	\sum_{r=0}^{\lfloor{\frac{n}{2}}\rfloor}  \Psi\left( \begin{array}{cc|r} a & b & n \\ \alpha & \beta & r \end{array} \right) \theta^r -
		\Psi(a-\alpha \theta , b - \beta \theta, n) \Big) q_1^{\lfloor{\frac{n}{2}}\rfloor} \pmod{\Lambda_\theta}
	\end{equation}
	As the congruence \eqref{two-II} is true for any $x,y$, and as $(\beta \theta - b) \alpha -  (\alpha \theta - a) \beta = \beta a - \alpha b \neq 0$, then from Lemma \eqref{SYMMETRIC2} the binary quadratic forms $\Lambda_\theta$ and $ q_1 $ are algebraic independent. This immediately leads to
			\begin{equation}
	\label{two-III}
	0= \sum_{r=0}^{\lfloor{\frac{n}{2}}\rfloor}  \Psi\left( \begin{array}{cc|r} a & b & n \\ \alpha & \beta & r \end{array} \right) \theta^r -
\Psi(a-\alpha \theta , b - \beta \theta, n).
	\end{equation}
	Hence we obtain the proof of \eqref{R3}. Similarly, we can prove \eqref{R3R}. \end{proof}

  The following desirable generalization is important

\begin{thm}
		\label{general}
	For any $a,b,\alpha,\beta, \eta, \xi, n$, $\beta a - \alpha b \neq 0,$ the following identities are true
	\begin{equation}
	\begin{aligned}
	\sum_{r=0}^{\lfloor{\frac{n}{2}}\rfloor}  \Psi\left( \begin{array}{cc|r} a & b & n \\ \alpha & \beta & r \end{array} \right)  \xi^{\lfloor{\frac{n}{2}}\rfloor  - r}  \eta^r =\Psi(a \xi-\alpha \eta, b \xi - \beta \eta, n), \\
		\sum_{r=0}^{\lfloor{\frac{n-1}{2}}\rfloor}  \Phi\left( \begin{array}{cc|r} a & b & n \\ \alpha & \beta & r \end{array} \right)  \xi^{\lfloor{\frac{n-1}{2}}\rfloor  - r}  \eta^r =\Phi(a \xi-\alpha \eta, b \xi - \beta \eta, n).
	\end{aligned}
	\end{equation}
\end{thm}
\begin{proof}
	Without loss of generality, let $\xi \neq 0$. We obtain the proof by replacing each $\theta$ in \eqref{R3} and  \eqref{R3R} by $\frac{\eta}{\xi},$ and multiplying each side of \eqref{R3} by $\xi^{\lfloor{\frac{n}{2}}\rfloor},$ and \eqref{R3R} by $\xi^{\lfloor{\frac{n-1}{2}}\rfloor},$  and noting that
	\[\xi^{\lfloor{\frac{n}{2}}\rfloor}\Psi(a-\alpha \frac{\eta}{\xi} , b - \beta \frac{\eta}{\xi}, n) = \Psi(a \xi-\alpha \eta, b \xi - \beta \eta, n) \] and
	 \[	\xi^{\lfloor{\frac{n-1}{2}}\rfloor}\Phi(a-\alpha \frac{\eta}{\xi} , b - \beta \frac{\eta}{\xi}, n) = \Phi(a \xi-\alpha \eta, b \xi - \beta \eta, n)
		 \]
	\end{proof}
Replacing $\theta$ by $\pm 1$ in \eqref{R3}, and \eqref{R3R}, we obtain the following desirable special cases
\begin{thm}
		\label{sum}
	For any $a,b,\alpha,\beta, n$, $\beta a - \alpha b \neq 0$ the following identities are true
	\begin{align}
	\begin{aligned}
	\label{theta equal 1}
	\centering	 	
	\sum_{r=0}^{\lfloor{\frac{n}{2}}\rfloor}  \Psi\left( \begin{array}{cc|r} a & b & n \\ \alpha & \beta & r \end{array} \right) &= \Psi(a-\alpha , b - \beta , n), \\ 	 	
	\sum_{r=0}^{\lfloor{\frac{n-1}{2}}\rfloor}  \Phi\left( \begin{array}{cc|r} a & b & n \\ \alpha & \beta & r \end{array} \right) &= \Phi(a-\alpha , b - \beta , n)
	\end{aligned}
	\end{align}
\end{thm}

\begin{thm}
	For any $a,b,\alpha,\beta, n$, $\beta a - \alpha b \neq 0$ the following identities are true
	\begin{align}
	\begin{aligned}
	\label{theta equal -11}
	\centering	 	
	\sum_{r=0}^{\lfloor{\frac{n}{2}}\rfloor}  \Psi\left( \begin{array}{cc|r} a & b & n \\ \alpha & \beta & r \end{array} \right) (-1)^{r} &= \Psi(a + \alpha , b + \beta , n),  \\
	\sum_{r=0}^{\lfloor{\frac{n-1}{2}}\rfloor}  \Phi\left( \begin{array}{cc|r} a & b & n \\ \alpha & \beta & r \end{array} \right) (-1)^{r} &= \Phi(a + \alpha , b + \beta , n)
	\end{aligned}
	\end{align}
\end{thm}
\subsection{GENERALIZATIONS FOR THEOREM \eqref{3R3}}
Now we can generalize Theorem \eqref{3R3} by applying the following specific differential map, for $k$ times, on \eqref{R3} and \eqref{R3R}, that sends
\[ a \longrightarrow 0 ,  b \longrightarrow 0, \alpha \longrightarrow 0 ,  \beta  \longrightarrow 0, \theta \longrightarrow -1  \] and noting that with respect to this differential map we get
\[ a-\alpha \theta \longrightarrow \alpha \quad , \quad b-\beta \theta \longrightarrow \beta.   \]
 Now, from \eqref{BASE1}, \eqref{BASE2} of Theorem \eqref{BASE} and from Corollary \eqref{VVV}, we get
 \begin{equation*}
 \begin{aligned}
  \Big( - \frac{{\partial}}{\partial \theta} \Big)^{k} &
  \Psi(a-\alpha \theta , b - \beta \theta, n) \\
  &=\Big( 0 \frac{{\partial} }{\partial a} +  0 \frac{{\partial} }{\partial b}   + 0  \frac{{\partial}}{\partial \alpha}  +  0  \frac{{\partial}}{\partial \beta} -  \frac{{\partial}}{\partial \theta}  \Big)^{k}
   \Psi\left( \begin{array}{cc|r} a-\alpha \theta & b - \beta \theta & n \\ \alpha & \beta & 0\end{array} \right) \\
 &= (-1)^{k} (k!) \Psi\left( \begin{array}{cc|r} a-\alpha \theta & b - \beta \theta & n \\ \alpha & \beta & k\end{array} \right), 	
 \end{aligned}
 \end{equation*}
   \begin{equation*}
 \begin{aligned}
 \Big( - \frac{{\partial}}{\partial \theta} \Big)^{k}&
 \Phi(a-\alpha \theta , b - \beta \theta, n) \\
  &= \Big( 0 \frac{{\partial} }{\partial a} +  0 \frac{{\partial} }{\partial b}   + 0  \frac{{\partial}}{\partial \alpha}  +  0  \frac{{\partial}}{\partial \beta} -  \frac{{\partial}}{\partial \theta}  \Big)^{k}
 \Phi\left( \begin{array}{cc|r} a-\alpha \theta & b - \beta \theta & n \\ \alpha & \beta & 0\end{array} \right) \\
 &= (-1)^{k} (k!) \Phi\left( \begin{array}{cc|r} a-\alpha \theta & b - \beta \theta & n \\ \alpha & \beta & k\end{array} \right). 	
 \end{aligned}
 \end{equation*}
  Hence we obtain the following desirable generalization for Theorem \eqref{3R3}.
  \begin{thm}
 	\label{RR}
 	For any $n,k, a,b,\alpha,\beta, \theta, n$, $\beta a - \alpha b \neq 0,$ the following identities are true
 	\begin{subequations}
 \begin{equation}
 	\label{RR1}
 	\sum_{r=k}^{\lfloor{\frac{n}{2}}\rfloor} \binom{r}{k}  \Psi\left( \begin{array}{cc|r} a & b & n \\ \alpha & \beta & r \end{array} \right) \theta^{r-k} = \Psi\left( \begin{array}{cc|r} a-\alpha \theta  & b-\beta \theta  & n \\ \alpha & \beta & k \end{array} \right),
 \end{equation}
 \begin{equation}
 \label{RR2}
 \sum_{r=k}^{\lfloor{\frac{n-1}{2}}\rfloor} \binom{r}{k}  \Phi\left( \begin{array}{cc|r} a & b & n \\ \alpha & \beta & r \end{array} \right) \theta^{r-k} = \Phi\left( \begin{array}{cc|r} a-\alpha \theta  & b-\beta \theta  & n \\ \alpha & \beta & k \end{array} \right)
  	\end{equation}
\end{subequations}
 \end{thm}

Again, without loss of generality, let $\xi \neq 0$. By replacing each $\theta$ in \eqref{RR1} and  \eqref{RR2} by $\frac{\eta}{\xi},$ and multiplying each side of \eqref{RR1} by $\xi^{\lfloor{\frac{n}{2}}\rfloor -k}$ and  \eqref{RR2} by $\xi^{\lfloor{\frac{n-1}{2}}\rfloor -k},$  and noting from the properties of $\Psi$ and $\Phi$ polynomials, given in \eqref{Wcommon3}, \eqref{Wcommon4} of Theorem \eqref{Wcommon}, that
\[ \xi^{\lfloor{\frac{n}{2}}\rfloor -k} \Psi\left( \begin{array}{cc|r} a-\alpha \frac{\eta}{\xi}  & b-\beta \frac{\eta}{\xi}  & n \\ \alpha & \beta & k \end{array}     \right)  =  \Psi\left( \begin{array}{cc|r} a\xi-\alpha \eta  & b\xi-\beta \eta  & n \\ \alpha & \beta & k \end{array} \right),         \]
\[ \xi^{\lfloor{\frac{n-1}{2}}\rfloor -k} \Phi\left( \begin{array}{cc|r} a-\alpha \frac{\eta}{\xi}  & b-\beta \frac{\eta}{\xi}  & n \\ \alpha & \beta & k \end{array}     \right)  =  \Phi\left( \begin{array}{cc|r} a\xi-\alpha \eta  & b\xi-\beta \eta  & n \\ \alpha & \beta & k \end{array} \right).         \]
Therefore, we obtain the following generalization for Theorem \eqref{RR}.
 \begin{thm}
	\label{RR-general}
	For any $n,k,a,b,\alpha,\beta, \theta, n$, $\beta a - \alpha b \neq 0,$ the following identities are true
	\begin{align}
	\begin{aligned}
	\label{RR1-general}
	\sum_{r=k}^{\lfloor{\frac{n}{2}}\rfloor} \binom{r}{k}  \Psi\left( \begin{array}{cc|r} a & b & n \\ \alpha & \beta & r \end{array} \right) \xi^{\lfloor{\frac{n}{2}}\rfloor  - r}  \eta^{r-k} = \Psi\left( \begin{array}{cc|r} a\xi-\alpha \eta  & b\xi-\beta \eta  & n \\ \alpha & \beta & k \end{array} \right),
	\end{aligned}
	\end{align}
		\begin{align}
	\begin{aligned}
	\label{RR2-general}
\sum_{r=k}^{\lfloor{\frac{n-1}{2}}\rfloor} \binom{r}{k}  \Phi\left( \begin{array}{cc|r} a & b & n \\ \alpha & \beta & r \end{array} \right) \xi^{\lfloor{\frac{n-1}{2}}\rfloor  - r}  \eta^{r-k} = \Phi\left( \begin{array}{cc|r} a\xi-\alpha \eta  & b\xi-\beta \eta  & n \\ \alpha & \beta & k \end{array} \right)
	\end{aligned}
	\end{align}
\end{thm}
\section{SPECIFIC FORMULAS FOR $\Psi$ AND $\Phi$ POLYNOMIALS }
Now, put $\xi = \alpha x^2 +\beta xy + \alpha y^2$ and $ \eta = a x^2 +bxy +a y^2$ in Theorem \eqref{general},  we obtain
\begin{equation}
\begin{aligned}
\label{W1}
\sum_{r=0}^{\lfloor{\frac{n}{2}}\rfloor}  \Psi\left( \begin{array}{cc|r} a & b & n \\ \alpha & \beta & r \end{array} \right) & (\alpha x^2 +\beta xy + \alpha y^2)^{\lfloor{\frac{n}{2}}\rfloor  - r}  (  a x^2 +bxy +a y^2 )^r \\
& = (  \beta a - \alpha b  )^{\lfloor{\frac{n}{2}}\rfloor} \Psi(xy,-x^2-y^2, n), \\
\end{aligned}
\end{equation}
\begin{equation}
\begin{aligned}
\label{W2}
\sum_{r=0}^{\lfloor{\frac{n-1}{2}}\rfloor}  \Phi\left( \begin{array}{cc|r} a & b & n \\ \alpha & \beta & r \end{array} \right) &  (\alpha x^2 +\beta xy + \alpha y^2)^{\lfloor{\frac{n-1}{2}}\rfloor  - r}  (  a x^2 +bxy +a y^2 )^r \\
&=(  \beta a - \alpha b  )^{\lfloor{\frac{n-1}{2}}\rfloor} \Phi(xy,-x^2-y^2, n).
\end{aligned}
\end{equation}
Now, from \eqref{special AAA}, and  \eqref{special BBB}, we get
\begin{equation}
\begin{aligned}
\label{W3}
\sum_{r=0}^{\lfloor{\frac{n}{2}}\rfloor}  \Psi\left( \begin{array}{cc|r} a & b & n \\ \alpha & \beta & r \end{array} \right) & (\alpha x^2 +\beta xy + \alpha y^2)^{\lfloor{\frac{n}{2}}\rfloor  - r}  (  a x^2 +bxy +a y^2 )^r \\
  &= (  \beta a - \alpha b  )^{\lfloor{\frac{n}{2}}\rfloor}\frac{x^n+y^n}{(x+y)^{\delta(n)}} ,
\end{aligned}
\end{equation}
and
\begin{equation}
\begin{aligned}
\label{W4}
\sum_{r=0}^{\lfloor{\frac{n-1}{2}}\rfloor}  \Phi\left( \begin{array}{cc|r} a & b & n \\ \alpha & \beta & r \end{array} \right) &  (\alpha x^2 +\beta xy + \alpha y^2)^{\lfloor{\frac{n-1}{2}}\rfloor  - r} (  a x^2 +bxy +a y^2 )^r \\
 &=(  \beta a - \alpha b  )^{\lfloor{\frac{n-1}{2}}\rfloor} \frac{x^n-y^n}{(x-y) (x+y)^{\delta(n-1)}}
\end{aligned}
\end{equation}
From \eqref{W1}, \eqref{W3}, we get formula \eqref{formula1} for $\Psi$ polynomial, and from \eqref{W2}, and \eqref{W4}, we get formula \eqref{formula2} for $\Phi$ polynomial.

\begin{thm}
	\label{formula}
	For any natural number $n$, the $\Psi$ and $\Phi$ polynomials satisfy the following polynomial identities
	\begin{equation}
	\label{formula1}
	\begin{aligned}
	\Psi(xy,-x^2-y^2,n) &= \frac{x^n+y^n}{(x+y)^{\delta(n)}},   \\
	\end{aligned}
	\end{equation}
	\begin{equation}
	\begin{aligned}
	\label{formula2}	
	\qquad	\qquad 		\Phi(xy,-x^2-y^2,n) &= \frac{x^n-y^n}{(x-y) (x+y)^{\delta(n-1)}}
	\end{aligned}
	\end{equation}
\end{thm}
 \section{FORMULA TO COMPUTE $ \Phi_r(n) $  DIRECTLY FROM $\Psi_r(n)$  }
 First, let \[ \Psi_r(n) :=\Psi\left( \begin{array}{cc|r} a & b & n \\ \alpha & \beta & r \end{array} \right) \quad , \quad  \Phi_r(n) :=\Phi\left( \begin{array}{cc|r} a & b & n \\ \alpha & \beta & r \end{array} \right).\]
  By differentiating ~\eqref{special AAA} with respect to the differential map
 \[ a \longrightarrow 0 ,  b \longrightarrow 0, \alpha \longrightarrow 0 ,  \beta  \longrightarrow 0,x \longrightarrow 1, y \longrightarrow -1,  \]
 and noting that with this differentiation we have
  \begin{align*}
 (\beta a - \alpha b)^{\lfloor{\frac{n}{2}}\rfloor} & \longrightarrow 0,  &  \Psi\left( \begin{array}{cc|r} a & b & n \\ \alpha & \beta & r \end{array} \right)   &\longrightarrow 0,   \\
  x^n + y^n       & \longrightarrow  n(x^{n-1} - y^{n-1}), & (x+y)^{\delta(n)} \longrightarrow  0, \\
    ax^2 +b xy +  ay^2    &\longrightarrow  (2a-b)(x-y),  &\alpha x^2 +\beta xy + \alpha y^2    &\longrightarrow  (2\alpha-\beta)(x-y),
 \end{align*}
 we get the following polynomial identity
  \begin{equation}
  \label{sahar1}
  \begin{aligned}
 &(\beta a - \alpha b)^{\lfloor{\frac{n}{2}}\rfloor} n \frac{(x^{n-1}-y^{n-1})}{(x+y)^{\delta(n)}}  = \\
  &\sum_{r=0}^{\lfloor{\frac{n}{2}}\rfloor - 1} (2\alpha-\beta) \Psi_r(n) (\lfloor{\frac{n}{2}}\rfloor -r)
 (\alpha x^2 + \beta xy + \alpha y^2)^{\lfloor{\frac{n}{2}}\rfloor -r -1} (ax^2+bxy+ay^{2})^{r} (x-y)\\
 + &\sum_{r=1}^{\lfloor{\frac{n}{2}}\rfloor} (2a-b) \Psi_r(n) (r)(\alpha x^2 + \beta xy + \alpha y^2)^{\lfloor{\frac{n}{2}}\rfloor -r} (ax^2+bxy+ay^{2})^{r-1} (x-y).
   \end{aligned}
  \end{equation}
 Raising each $n$ by one, and noting that $  (\beta a - \alpha b)^{\lfloor{\frac{n+1}{2}}\rfloor} =  (\beta a - \alpha b)^{\lfloor{\frac{n-1}{2}}\rfloor}  (\beta a - \alpha b), \lfloor{\frac{n+1}{2}}\rfloor -1 =  \lfloor{\frac{n-1}{2}}\rfloor, \delta(n+1)=\delta(n-1)$, and then dividing by $(n+1)(x-y)$, we get the following formula
   \begin{equation}
 \label{sahar2}
 \begin{aligned}
 &(\beta a - \alpha b)^{\lfloor{\frac{n-1}{2}}\rfloor} \frac{(x^{n}-y^{n})}{(x-y)(x+y)^{\delta(n)}}  = \\
 &\sum_{r=0}^{\lfloor{\frac{n-1}{2}}\rfloor} \frac{(2\alpha -\beta)( \lfloor{\frac{n+1}{2}}\rfloor     -r) \Psi_r(n+1) }{(\beta a - \alpha b)(n+1)}
 (\alpha x^2 + \beta xy + \alpha y^2)^{\lfloor{\frac{n-1}{2}}\rfloor -r } (ax^2+bxy+ay^{2})^{r} \\
 +&\sum_{r=0}^{\lfloor{\frac{n-1}{2}}\rfloor} \frac{ (2a-b)(r+1) \Psi_{r+1}(n+1) }{(\beta a - \alpha b)(n+1)}
 (\alpha x^2 + \beta xy + \alpha y^2)^{\lfloor{\frac{n-1}{2}}\rfloor -r } (ax^2+bxy+ay^{2})^{r}.
 \end{aligned}
 \end{equation}
 Now, comparing \eqref{sahar2} with \eqref{special BBB}, we immediately get the following desirable formula.
  \begin{thm}
 	\label{BBB from AAA }
 	For any natural number $n$, $\beta a - \alpha b \neq 0 $,
 	the following formula is true
 	\begin{equation}
 	\label{BBB-AAA}	
 	 	\Phi_r(n) = \frac{(2\alpha -\beta)( \lfloor{\frac{n+1}{2}}\rfloor     -r) \Psi_r(n+1) + (2a-b)(r+1) \Psi_{r+1}(n+1)          }{(\beta a - \alpha b)(n+1)}
 	\end{equation}
 \end{thm}

\section{THE $\Psi$ AND $\Phi$ TRAJECTORIES AND ORBITS OF ORDER $n$}
From Theorems \eqref{BASE}, \eqref{BASE3}, \eqref{Wcommon}, we know that for any numbers $a,b,\alpha, \beta$, the following relations are true for any natural number $n$
	\begin{equation*}
\label{orbits1}
\begin{aligned}
\Psi_0(n)  &= \Psi\left( \begin{array}{cc|r} a & b & n \\ \alpha & \beta & 0 \end{array} \right) = \Psi(a,b,n),   \\
 \Psi_{\lfloor{\frac{n}{2}}\rfloor}(n)&=  \Psi\left( \begin{array}{cc|c} a & b & n \\ \alpha & \beta & \lfloor{\frac{n}{2}}\rfloor \end{array} \right) = (-1)^{\lfloor{\frac{n}{2}}\rfloor} \Psi(\alpha,\beta,n) = \Psi(-\alpha,-\beta,n), \\ 
 \Phi_0(n)&= \Phi\left( \begin{array}{cc|r} a & b & n \\ \alpha & \beta & 0 \end{array} \right) = \Phi(a,b,n),  \\
 \Phi_{\lfloor{\frac{n-1}{2}}\rfloor}(n)&=  \Phi\left( \begin{array}{cc|c} a & b & n \\ \alpha & \beta & \lfloor{\frac{n-1}{2}}\rfloor \end{array}\right) =(-1)^{\lfloor{\frac{n-1}{2}}\rfloor} \Phi(\alpha,\beta,n) = \Phi(-\alpha,-\beta,n).
\end{aligned}
\end{equation*}

Therefore, Theorems \eqref{special AA}, \eqref{BASE}, \eqref{orbits} should motivate us define the trajectories which are constructed from the sequence of polynomials $\Psi_r(n)$ and $\Phi_r(n)$ which connects $\Psi(a,b,n)$ with $\Psi(-\alpha,-\beta,n)$ and $\Phi(a,b,n)$ with $\Phi(-\alpha,-\beta,n)$ respectively. 

\begin{defn}
	Let $a,b, \alpha, \beta$ be given real numbers, $\beta a - \alpha b \neq 0$,  and \[ \Psi_r(n) := \Psi\left( \begin{array}{cc|r} a & b & n \\ \alpha & \beta & r \end{array} \right) \quad , \quad \Phi_r(n) := \Phi\left( \begin{array}{cc|r} a & b & n \\ \alpha & \beta & r \end{array} \right).\] We call the sequence of polynomials  
	\begin{equation}
	\label{orbits1}
\Psi_0(n), \Psi_1(n), \Psi_2(n), \dots , \Psi_{\lfloor{\frac{n}{2}}\rfloor}(n) 
	\end{equation}
the $\Psi$ trajectory of order $n$ from $(a,b)$ to $(\alpha,\beta)$.  Alternatively, we call that the $\Psi$ trajectory \say{connect}  $\Psi(a,b,n)$ with $\Psi(-\alpha,-\beta,n)$. When  $\Psi(-\alpha,-\beta,n)=\Psi(a,b,n)$ for given natural number $n$ then we call the $\Psi$ trajectory of order $n$ from $(a,b)$ to $(\alpha,\beta)$ by the $\Psi$ orbit of order $n$ from $(a,b)$ to $(\alpha,\beta)$.
Similarly, we call the sequence of polynomials  
	\begin{equation}
	\label{orbits2}
\Phi_0(n), \Phi_1(n), \Phi_2(n), \dots , \Phi_{\lfloor{\frac{n-1}{2}}\rfloor}(n)  
	\end{equation}
the $\Phi$ trajectory of order $n$ from $(a,b)$ to $(\alpha,\beta)$.  Alternatively, we call that the $\Phi$ trajectory \say{connect}  $\Phi(a,b,n)$ with $\Phi(-\alpha,-\beta,n)$. When  $\Phi(-\alpha,-\beta,n)=\Phi(a,b,n)$ for given natural number $n$ then we call the $\Phi$ trajectory of order $n$ from $(a,b)$ to $(\alpha,\beta)$ by the $\Phi$ orbit of order $n$ from $(a,b)$ to $(\alpha,\beta)$. 
	\end{defn} 
\subsection*{MAJOR QUESTIONS ABOUT THE $\Psi$ AND $\Phi$ TRAJECTORIES AND ORBITS}
 Major focus of future research on the $\Psi$ and $\Phi$ trajectories and orbits are the following questions 
 \begin{enumerate}
	\item Does an efficient algorithms exist to identify whether a given $\Psi$ and $\Phi$ trajectory can be turned to an orbit ?
	\item What are the common arithmetical characteristics of the terms of the $\Psi$ and $\Phi$ orbit?
	\item  What are the common properties of the prime factors of the terms of $\Psi$ and $\Phi$ orbit?
		\item Does the arithmetic of the integer terms of the Fermat orbit help finding an efficient algorithms to factorize Fermat numbers ? 
	
\end{enumerate}

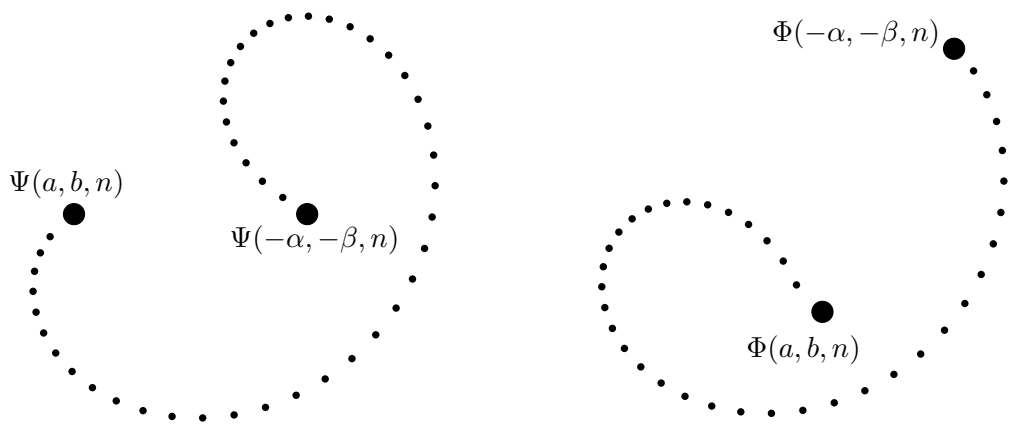
\makeatletter
\define@key{cylindricalkeys}{angle}{\def\myangle{#1}}
\define@key{cylindricalkeys}{radius}{\def\myradius{#1}}
\define@key{cylindricalkeys}{z}{\def\myz{#1}}
\tikzdeclarecoordinatesystem{cylindrical}%
{%
	\setkeys{cylindricalkeys}{#1}%
	\pgfpointadd{\pgfpointxyz{0}{0}{\myz}}{\pgfpointpolarxy{\myangle}{\myradius}}
}
\begin{tikzpicture}[z=0.2pt]
\node (a) at (2.2,5.5) {$\Psi(a,b,n)$};
\node (a) at (5.5,4.759) {$\Psi(-\alpha,-\beta,n)$};
\filldraw [] 
(2.3,5.1) circle (4pt);
\filldraw [] 
(5.4,5.1)  circle (4pt);
\foreach \num in {500,510,...,950}
\fill (cylindrical cs:angle=\num,radius=2,z=\num) circle (1.4pt);
\node [right=1cm,text width=15cm,font=\footnotesize] 
{ The notions of $\Psi$ and $\Phi$ trajectories arose up naturally in this paper. Many well-known sequences are now not isolated; meaning we can connect them by a certain trajectory.  When the endpoint $\Psi(-\alpha,-\beta,n)$ coincidences with starting point $\Psi(a,b,n)$, then the trajectory gives the $\Psi$ orbit. Also, when the endpoint $\Phi(-\alpha,-\beta,n)$ coincidences with starting point $\Phi(a,b,n)$, then the trajectory gives the $\Phi$ orbit. These trajectories and orbits might be a new area of research in the future. 
};
\tikzdeclarecoordinatesystem{cylindrical}%
{%
	\setkeys{cylindricalkeys}{#1}%
	\pgfpointadd{\pgfpointxyz{2.5}{-5}{\myz}}{\pgfpointpolarxy{\myangle}{\myradius}}
}
\node (a) at (12,3.3) {$\Phi(a,b,n)$};
\node (a) at (12.7,7.5) {$\Phi(-\alpha,-\beta,n)$};
\filldraw [] 
(12.25,3.8) circle (4pt);
\filldraw [] 
(14,7.3) circle (4pt);

\foreach \num in {1120,1130,...,1500}
\fill (cylindrical cs:angle=\num,radius=2,z=\num) circle (1.4pt);
\node [right=1cm,text width=12cm,font=\footnotesize] 
{ 
};
\end{tikzpicture}

From the relations \eqref{BASE1}, \eqref{BASE2} of Theorem \eqref{BASE}, we should have no difficulties to prove the following relations of the following theorems.
\begin{thm}
	\label{important}
	For any real numbers $a,b$, $a \neq \pm b$, and any natural number $n$, the following identities are true
		\begin{equation*}
	\label{important1}
	\begin{aligned}
	\Psi(a,-b,n)&=\Psi(-a,-b,n) \qquad   \mbox{for $n$ even}       \\
	 \Psi(a,-b,n)&= \Phi(-a,-b,n)  \qquad  \mbox{for $n$ odd}       \\
	 \Phi(a,-b,n)&=\Phi(-a,-b,n)    \qquad    \mbox{for $n$ even}         \\
	 \Phi(a,-b,n)&= \Psi(-a,-b,n) \qquad \mbox {for $n$ odd}  \\
	\end{aligned}
	\end{equation*}
\end{thm}

\begin{thm}
	\label{important2}
	For any real numbers $a,b$, $a \neq \pm b$, and any natural number $n$, the following identity is true
	\begin{equation*}
	\label{important3}
	\begin{aligned}
	\Phi(a,b,2n)&= \Phi(a,b,n) \Psi(a,b,n) \qquad  
	\end{aligned}
	\end{equation*}
\end{thm}

\subsection*{EXAMPLES FOR TRAJECTORIES CONNECT WELL-KNOWN SEQUENCES}
We already proved in the current paper that the $\Psi$ and $\Phi$ polynomials unify many well-known polynomials and sequences. Now the notions of the $\Psi$ and $\Phi$ trajectories should be useful. 
			
		\subsection{The CHEBYSHEV-LUCAS TRAJECTORY} 

The $\Psi$ trajectory of order $n$ from $(1,2-4x^2)$ to $(1,3)$ starts with $\frac{2^{\delta(n+1)}}{x^{\delta(n)}}T_n(x)$ and ends with  $L(n)$. So this way we get new path connects Chebyshev polynomials of the first kind with Lucas sequences. The vision behind considering such trajectory is that this path should help to understand some characteristics of the starting party, the Chebyshev Polynomials, by knowing some characteristics of the final party, the Lucas numbers, and vice versa.\\

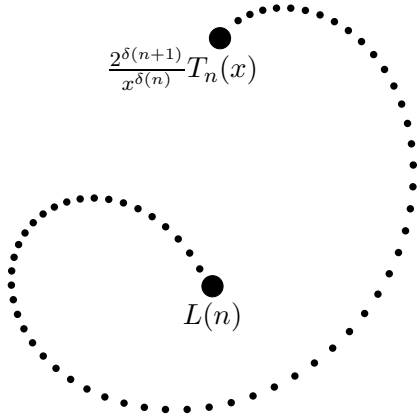
\makeatletter
\define@key{cylindricalkeys}{angle}{\def\myangle{#1}}
\define@key{cylindricalkeys}{radius}{\def\myradius{#1}}
\define@key{cylindricalkeys}{z}{\def\myz{#1}}
\tikzdeclarecoordinatesystem{cylindrical}%
{%
	\setkeys{cylindricalkeys}{#1}%
	\pgfpointadd{\pgfpointxyz{0}{0}{\myz}}{\pgfpointpolarxy{\myangle}{\myradius}}
}
\begin{tikzpicture}[z=0.2pt]
\node (a) at (1.9,1.1) {$L(n)$};
\node (b) at (1.5,4.4) {$\frac{2^{\delta(n+1)}}{x^{\delta(n)}}T_n(x)$};
\filldraw [] 
(2,4.8) circle (4pt);
\filldraw [] 
(1.9,1.5) circle (4pt);
\foreach \num in {45,52,...,490}
\fill (cylindrical cs:angle=\num,radius=2,z=\num) circle (1.4pt);
\node (c) at (4.2,2.2)  [right=1cm,text width=9cm,font=\footnotesize] 
{Studying this particular trajectory should help understand the arithmetic of Chebyshev polynomials of the first kind and Lucas numbers. For any natural number $n$, we call this path the Chebyshev-Lucas trajectory which is the sequence
	\[ \frac{2^{\delta(n+1)}}{x^{\delta(n)}}T_n(x) = \Psi_0(n), \Psi_1(n), \Psi_2(n), \dots, \Psi_{\lfloor{\frac{n}{2}}\rfloor}(n) =  L(n)  \]
	where $\Psi_r(n) = \Psi\left( \begin{array}{cc|r} 1 & 2-4x^2 & n \\ 1 & 3 & r \end{array} \right)$. Moreover, the terms $\Psi_r(n)$ satisfy the following polynomial identity 
	\[(1+4x^2)^{\lfloor{\frac{n}{2}}\rfloor}\frac{z^n+t^n}{(z+t)^{\delta(n)}} = \] 
	\[  \sum_{r=0}^{\lfloor{\frac{n}{2}}\rfloor}
	\Psi_r(n) ( z^2 + 3zt +  t^2)^{\lfloor{\frac{n}{2}}\rfloor -r} (z^2+   (2-4x^2)zt+t^{2})^{r}. \] 
};
\end{tikzpicture}
		\subsection{THE LUCAS-FIBONACCI TRAJECTORY} 
Let $n$ be any odd natural number. The $\Psi$ trajectory of order $n$ from $(-1,-3)$ to $(-1,+3)$ starts with $L(n)$ and ends with $F(n)$. So this way we get path connects Lucas numbers with Fibonacci numbers.

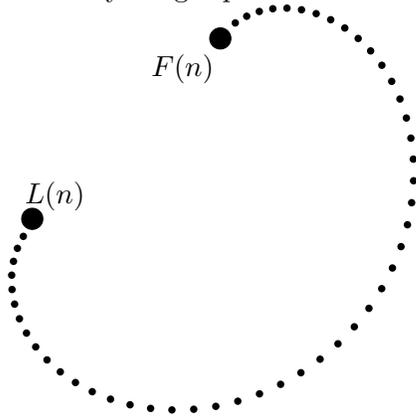
\makeatletter
\define@key{cylindricalkeys}{angle}{\def\myangle{#1}}
\define@key{cylindricalkeys}{radius}{\def\myradius{#1}}
\define@key{cylindricalkeys}{z}{\def\myz{#1}}
\tikzdeclarecoordinatesystem{cylindrical}%
{%
	\setkeys{cylindricalkeys}{#1}%
	\pgfpointadd{\pgfpointxyz{0}{0}{\myz}}{\pgfpointpolarxy{\myangle}{\myradius}}
}
\begin{tikzpicture}[z=0.2pt]
\node (a) at (-0.2,2.7) {$L(n)$};
\node (b) at (1.5,4.4) {$F(n)$};
\filldraw [] 
(2,4.8) circle (4pt);
\filldraw [] 
(-0.5,2.4) circle (4pt);
\foreach \num in {145,152,...,490}
\fill (cylindrical cs:angle=\num,radius=2,z=\num) circle (1.4pt);
\node (c) at (4.2,2.2)  [right=1cm,text width=9cm,font=\footnotesize] 
{Studying this particular trajectory should enjoy certain properties that comes from the arithmetic of Lucas numbers and Fibonacci. For any odd natural number $n$, we call this path the Lucas-Fibonacci trajectory which is the sequence
	\[L(n) = \Psi_0(n), \Psi_1(n), \Psi_2(n), \dots, \Psi_{\lfloor{\frac{n}{2}}\rfloor}(n) =  F(n)  \]
	where $\Psi_r(n) = \Psi\left( \begin{array}{cc|r} -1 & -3 & n \\ -1 & +3 & r \end{array} \right)$. Moreover, the terms $\Psi_r(n)$ satisfy the following polynomial identity 
	\[(-6)^{\lfloor{\frac{n}{2}}\rfloor}\frac{z^n+t^n}{(z+t)^{\delta(n)}} = \] 
	\[  \sum_{r=0}^{\lfloor{\frac{n}{2}}\rfloor}
	\Psi_r(n) ( -z^2 + 3zt - t^2)^{\lfloor{\frac{n}{2}}\rfloor -r} (-z^2 -3zt-t^{2})^{r}. \] 
};
\end{tikzpicture}

	\subsection{THE LUCAS ORBIT} 
Let $n$ be any even natural number. The $\Psi$ trajectory of order $n$ from $(-1,-3)$ to $(-1,+3)$ starts with $L(n)$ and ends with $L(n)$. So this way we get orbit. We call this orbit the Lucas orbit at level $n$.  \\

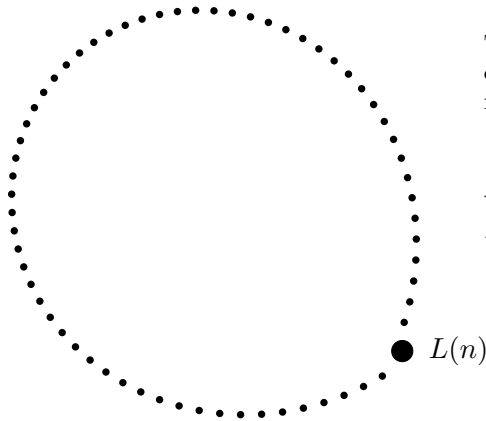
\makeatletter
\define@key{cylindricalkeys}{angle}{\def\myangle{#1}}
\define@key{cylindricalkeys}{radius}{\def\myradius{#1}}
\define@key{cylindricalkeys}{z}{\def\myz{#1}}
\tikzdeclarecoordinatesystem{cylindrical}%
{%
	\setkeys{cylindricalkeys}{#1}%
	\pgfpointadd{\pgfpointxyz{0}{0}{\myz}}{\pgfpointpolarxy{\myangle}{\myradius}}
}
\begin{tikzpicture}[z=0.1pt]
\node (a) at (6.1,1.32) {$L(n) $};
\filldraw [] 
(5.35,1.32) circle (4pt);
\foreach \num in {705,710,...,1020}
\fill (cylindrical cs:angle=\num,radius=3,z=\num) circle (1.4pt);
\node (b) at (5.6,3.0)  [right=0.7cm,text width=9cm,font=\footnotesize] 
{This particular orbit should enjoy unexpected properties that comes from the arithmetic of Lucas numbers. For any even natural number $n$, the Lucas orbit at level $n$ is the sequence
	\[L(n) = \Psi_0(n), \Psi_1(n), \Psi_2(n), \dots, \Psi_{\lfloor{\frac{n}{2}}\rfloor}(n) =  L(n)  \]
	where $\Psi_r(n) = \Psi\left( \begin{array}{cc|r} -1 & -3 & n \\ -1 & +3 & r \end{array} \right)$. Moreover, the terms $\Psi_r(n)$ satisfy the following polynomial identity 
	\[(-6)^{\lfloor{\frac{n}{2}}\rfloor}\frac{z^n+t^n}{(z+t)^{\delta(n)}} = \] 
	\[  \sum_{r=0}^{\lfloor{\frac{n}{2}}\rfloor}
	\Psi_r(n) ( -z^2 + 3zt - t^2)^{\lfloor{\frac{n}{2}}\rfloor -r} (-z^2 -3zt-t^{2})^{r}. \]};
\end{tikzpicture}

	\subsection{THE LUCAS-PELL TRAJECTORY} 
Let $n$ be any natural number. The $\Psi$ trajectory of order $n$ from  $(-1,-3)$ to $(1,6)$ starts with $L(n)$ and ends with $\frac{Q_n}{2^{\delta(n)}}$, where $Q_n$ are the Pell-Lucas numbers. \\

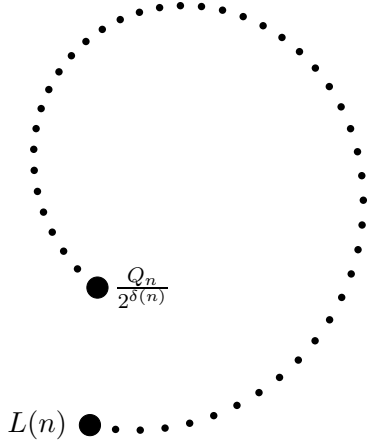
\makeatletter
\define@key{cylindricalkeys}{angle}{\def\myangle{#1}}
\define@key{cylindricalkeys}{radius}{\def\myradius{#1}}
\define@key{cylindricalkeys}{z}{\def\myz{#1}}
\tikzdeclarecoordinatesystem{cylindrical}%
{%
	\setkeys{cylindricalkeys}{#1}%
	\pgfpointadd{\pgfpointxyz{0}{0}{\myz}}{\pgfpointpolarxy{\myangle}{\myradius}}
}
\begin{tikzpicture}[z=0.1pt]
\node (a) at (4.5,3.57) {$L(n) $};
\node (a) at (5.9,5.4) {$\frac{Q_n}{2^{\delta(n)}}$};
\filldraw [] 
(5.3,5.4) circle (4pt);
\filldraw [] 
(5.2,3.57) circle (4pt);
\foreach \num in {1700,1707,...,2020}
\fill (cylindrical cs:angle=\num,radius=2.5,z=\num) circle (1.4pt);
\node (b) at (8.3,5.8)  [right=1.2cm,text width=9cm,font=\footnotesize] 
{So, we get trajectory passing through the Lucas numbers and Lucas-Pell numbers. We call this trajectory the Lucas-Pell trajectory. This trajectory should enjoy some properties that comes from the arithmetic of Lucas and Lucas-Pell numbers. For any natural number $n$, the Lucas-Pell trajectory at level $n$ is the sequence
	\[L(n) = \Psi_0(n), \Psi_1(n), \Psi_2(n), \dots, \Psi_{\lfloor{\frac{n}{2}}\rfloor}(n) =  \frac{Q_n}{2^{\delta(n)}}  \]
	where $\Psi_r(n) = \Psi\left( \begin{array}{cc|r} -1 & -3 & n \\ 1 & 6 & r \end{array} \right)$. Moreover, the terms $\Psi_r(n)$ satisfy the following polynomial identity 
	\[(-3)^{\lfloor{\frac{n}{2}}\rfloor}\frac{z^n+t^n}{(z+t)^{\delta(n)}} = \] 
	\[  \sum_{r=0}^{\lfloor{\frac{n}{2}}\rfloor}
	\Psi_r(n) ( z^2 + 6zt + t^2)^{\lfloor{\frac{n}{2}}\rfloor -r} (-z^2 -3zt-t^{2})^{r}. \]};
\end{tikzpicture}

\subsection{THE FIBONACCI-PELL TRAJECTORY} 
Let $n$ be any natural number. The $\Phi$ trajectory of order $n$ from  $(-1,-3)$ to $(1,6)$ starts with $F(n)$ and ends with $\frac{P_n}{2^{\delta(n-1)}}$, where $P_n$ are the Pell numbers. \\

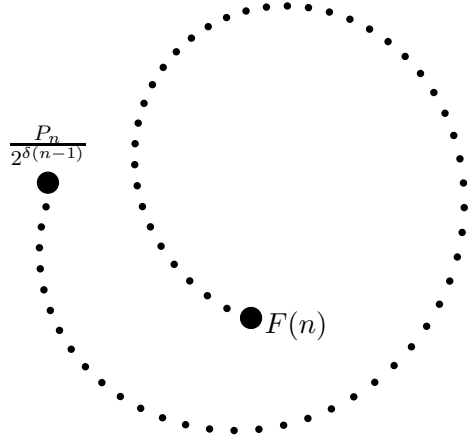
\makeatletter
\define@key{cylindricalkeys}{angle}{\def\myangle{#1}}
\define@key{cylindricalkeys}{radius}{\def\myradius{#1}}
\define@key{cylindricalkeys}{z}{\def\myz{#1}}
\tikzdeclarecoordinatesystem{cylindrical}%
{%
	\setkeys{cylindricalkeys}{#1}%
	\pgfpointadd{\pgfpointxyz{0}{0}{\myz}}{\pgfpointpolarxy{\myangle}{\myradius}}
}
\begin{tikzpicture}[z=0.1pt]
\node (a) at 	(6.6,4.9) {$F(n) $};
\node (a) at 	(3.3,7.3) {$\frac{P_n}{2^{\delta(n-1)}}$};
\filldraw [] 
(3.3,6.8) circle (4pt);
\filldraw [] 
(6.0,5.0) circle (4pt);
\foreach \num in {1600,1607,...,2040}
\fill (cylindrical cs:angle=\num,radius=2.5,z=\num) circle (1.4pt);
\node (b) at (8.3,5.8)  [right=1cm,text width=9cm,font=\footnotesize] 
{So, we get trajectory passing through the Lucas numbers and Lucas-Pell numbers. We call this trajectory the Lucas-Pell trajectory. This trajectory should enjoy some properties that comes from the arithmetic of Lucas and Lucas-Pell numbers. For any natural number $n$, the Lucas-Pell trajectory at level $n$ is the sequence
	\[F(n) = \Phi_0(n), \Phi_1(n), \Phi_2(n), \dots, \Phi_{\lfloor{\frac{n-1}{2}}\rfloor}(n) =  \frac{P_n}{2^{\delta(n-1)}}  \]
	where $\Phi_r(n) = \Phi\left( \begin{array}{cc|r} -1 & -3 & n \\ 1 & 6 & r \end{array} \right)$. Moreover, the terms $\Phi_r(n)$ satisfy the following polynomial identity 
	\[(-3)^{\lfloor{\frac{n-1}{2}}\rfloor}\frac{z^n-t^n}{(z-t)(z+t)^{\delta(n-1)}} = \] 
	\[  \sum_{r=0}^{\lfloor{\frac{n-1}{2}}\rfloor}
	\Phi_r(n) ( z^2 + 6zt + t^2)^{\lfloor{\frac{n-1}{2}}\rfloor -r} (-z^2 -3zt-t^{2})^{r}. \]};
\end{tikzpicture}

\subsection{THE FIBONACCI ORBIT} 
Let $n$ be any even natural number. The $\Phi$ trajectory of order $n$ from $(-1,-3)$ to $(-1,+3)$ starts with $F_n$ and ends with $F_n$, where $F_n$ are the Fibonacci numbers. So, we get orbit passing through the Fibonacci numbers. We call this orbit the Fibonacci orbit at level $n$.  \\

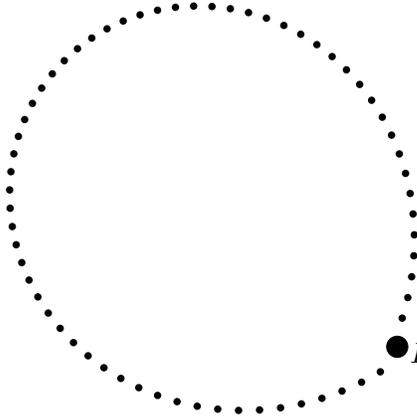
\makeatletter
\define@key{cylindricalkeys}{angle}{\def\myangle{#1}}
\define@key{cylindricalkeys}{radius}{\def\myradius{#1}}
\define@key{cylindricalkeys}{z}{\def\myz{#1}}
\tikzdeclarecoordinatesystem{cylindrical}%
{%
	\setkeys{cylindricalkeys}{#1}%
	\pgfpointadd{\pgfpointxyz{0}{0}{\myz}}{\pgfpointpolarxy{\myangle}{\myradius}}
}
\begin{tikzpicture}[z=0.1pt]
\node (a) at (5.9,1.2) {$F(n)$};
\filldraw [] 
(5.31,1.32) circle (4pt);
\foreach \num in {705,710,...,1020}
\fill (cylindrical cs:angle=\num,radius=3,z=\num) circle (1.4pt);
\node (b) at (5.6,3.2)  [right=0.5cm,text width=9cm,font=\footnotesize] 
{Therefore, for any even natural number $n$, the Fibonacci orbit at level $n$ is the sequence
	\[F_n = \Phi_0(n), \Phi_1(n), \Phi_2(n), \dots, \Phi_{\lfloor{\frac{n}{2}}\rfloor}(n) = F_n\]
	where $\Phi_r(n) = \Phi\left( \begin{array}{cc|r} -1 & -3 & n \\ -1 & +3 & r \end{array} \right)$. Moreover, the terms $\Phi_r(n)$ satisfy the following polynomial identity 
	\[(-6)^{\lfloor{\frac{n-1}{2}}\rfloor}\frac{z^n-t^n}{z^2-t^2} = \] 
	\[  \sum_{r=0}^{\lfloor{\frac{n-1}{2}}\rfloor}
	\Phi_r(n) ( -z^2+3zt - t^2)^{\lfloor{\frac{n-1}{2}}\rfloor -r} (-z^2 -3zt-t^{2})^{r}. \]};
\end{tikzpicture}
\subsection{THE FIBONACCI-LUCAS TRAJECTORY} 
Let $n$ be any odd natural number. The $\Phi$ trajectory of order $n$ from $(-1,-3)$ to $(-1,+3)$ starts with $F(n)$ and ends with $L(n)$. So this way we get another new path connects Fibonacci numbers with Lucas numbers. \\

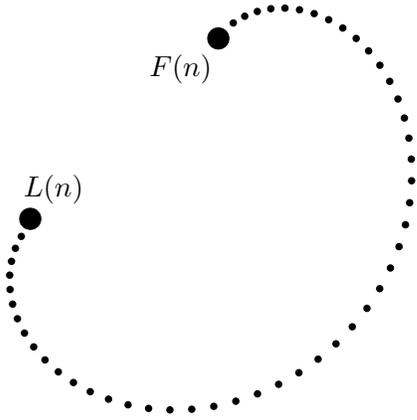
\makeatletter
\define@key{cylindricalkeys}{angle}{\def\myangle{#1}}
\define@key{cylindricalkeys}{radius}{\def\myradius{#1}}
\define@key{cylindricalkeys}{z}{\def\myz{#1}}
\tikzdeclarecoordinatesystem{cylindrical}%
{%
	\setkeys{cylindricalkeys}{#1}%
	\pgfpointadd{\pgfpointxyz{0}{0}{\myz}}{\pgfpointpolarxy{\myangle}{\myradius}}
}
\begin{tikzpicture}[z=0.2pt]
\node (a) at (-0.19,2.8) {$L(n)$};
\node (b) at (1.5,4.4) {$F(n)$};
\filldraw [] 
(2,4.8) circle (4pt);
\filldraw [] 
(-0.5,2.4) circle (4pt);
\foreach \num in {145,152,...,490}
\fill (cylindrical cs:angle=\num,radius=2,z=\num) circle (1.4pt);
\node (c) at (4.2,2.2)  [right=1cm,text width=9cm,font=\footnotesize] 
{ Therefore, for any odd natural number $n$, we call this path the Lucas-Fibonacci trajectory which is the sequence
	\[F(n) = \Phi_0(n), \Phi_1(n), \Phi_2(n), \dots, \Phi_{\lfloor{\frac{n}{2}}\rfloor}(n) =  L(n)  \]
	where $\Phi_r(n) = \Phi\left( \begin{array}{cc|r} -1 & -3 & n \\ -1 & +3 & r \end{array} \right)$. Moreover, the terms $\Phi_r(n)$ satisfy the following polynomial identity 
	\[(-6)^{\lfloor{\frac{n-1}{2}}\rfloor}\frac{z^n-t^n}{z-t} = \] 
	\[  \sum_{r=0}^{\lfloor{\frac{n-1}{2}}\rfloor}
	\Phi_r(n) ( -z^2 + 3zt - t^2)^{\lfloor{\frac{n-1}{2}}\rfloor -r} (-z^2 -3zt-t^{2})^{r}. \] 
};
\end{tikzpicture}

\subsection*{OBSERVATION}
For any odd natural number $n$, we should observe that the sequence
	\[F(n) = \Phi_0(n), \Phi_1(n), \dots, \Phi_{\lfloor{\frac{n}{2}}\rfloor}(n) =  L(n) = \Psi_0(n), \Psi_1(n), \dots, \Psi_{\lfloor{\frac{n}{2}}\rfloor}(n) =  F(n)   \]
is orbit passing through both of Fibonacci and Lucas numbers, where $\Phi_r(n) = \Phi\left( \begin{array}{cc|r} -1 & -3 & n \\ -1 & +3 & r \end{array} \right)$ and $\Psi_r(n) = \Psi\left( \begin{array}{cc|r} -1 & -3 & n \\ -1 & +3 & r \end{array} \right)$. We call such orbit the \say{Fibonacci-Lucas orbit}. We should ask for explicit closed formula for the terms of this orbit. The arithmetical characteristics for this orbit should be intimately related with the Fibonacci and Lucas numbers.

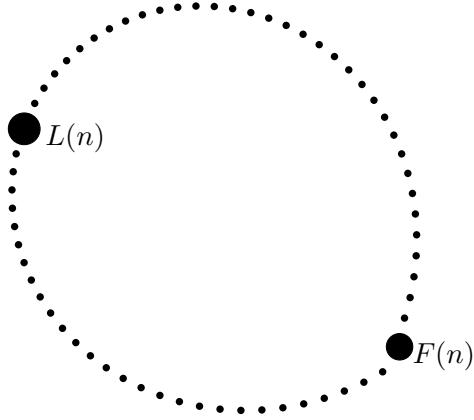
\makeatletter
\define@key{cylindricalkeys}{angle}{\def\myangle{#1}}
\define@key{cylindricalkeys}{radius}{\def\myradius{#1}}
\define@key{cylindricalkeys}{z}{\def\myz{#1}}
\tikzdeclarecoordinatesystem{cylindrical}%
{%
	\setkeys{cylindricalkeys}{#1}%
	\pgfpointadd{\pgfpointxyz{0}{0}{\myz}}{\pgfpointpolarxy{\myangle}{\myradius}}
}
\begin{tikzpicture}[z=0.1pt]
\node (a) at (5.9,1.2) {$F(n)$};
\filldraw [] 
(5.31,1.32) circle (5pt);
\node (a) at (1,4.1) {$L(n)$};
\filldraw [] 
(0.32,4.22) circle (6pt);

\foreach \num in {705,710,...,1020}
\fill (cylindrical cs:angle=\num,radius=3,z=\num) circle (1.4pt);
\node (b) at (5.9,3.5)  [right=1cm,text width=7.3cm,font=\footnotesize] 
{ The Fibonacci-Lucas trajectory of order $n$ from $(-1,-3)$ to $(-1,+3)$ and the Lucas-Fibonacci trajectory of order $n$ from $(-1,-3)$ to $(-1,+3)$ can be combined to construct a new orbit that go through both of the Fibonacci and Lucas numbers.  };
\end{tikzpicture}

		\subsection{THE MERSENNE ORBIT} 
Let $n$ be any even natural number. The $\Phi$ trajectory of order $n$ from $(2,-5)$ to $(2,+5)$ starts with $\frac{M_n}{3}$ and ends with $\frac{M_n}{3}$, where $M_n$ are the Mersenne numbers defined by $M_n=2^n-1$. So, we get orbit. We call this orbit the Mersenne orbit at level $n$.  \\

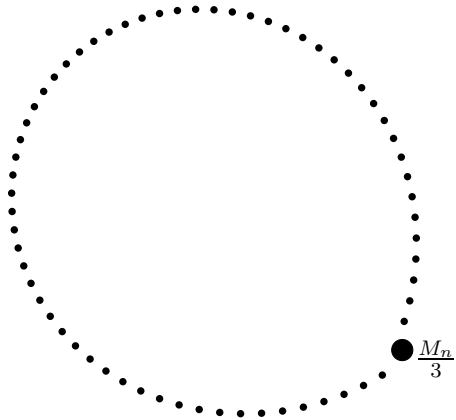
\makeatletter
\define@key{cylindricalkeys}{angle}{\def\myangle{#1}}
\define@key{cylindricalkeys}{radius}{\def\myradius{#1}}
\define@key{cylindricalkeys}{z}{\def\myz{#1}}
\tikzdeclarecoordinatesystem{cylindrical}%
{%
	\setkeys{cylindricalkeys}{#1}%
	\pgfpointadd{\pgfpointxyz{0}{0}{\myz}}{\pgfpointpolarxy{\myangle}{\myradius}}
}
\begin{tikzpicture}[z=0.1pt]
\node (a) at (5.8,1.21) {$\frac{M_n}{3}$};
\filldraw [] 
(5.35,1.32) circle (4pt);
\foreach \num in {705,710,...,1020}
\fill (cylindrical cs:angle=\num,radius=3,z=\num) circle (1.4pt);
\node (b) at (5.6,3.2)  [right=0.5cm,text width=9cm,font=\footnotesize] 
{For any even natural number $n$, the Mersenne orbit at level $n$ is the sequence
	\[\frac{M_n}{3} = \Phi_0(n), \Phi_1(n), \Phi_2(n), \dots, \Phi_{\lfloor{\frac{n}{2}}\rfloor}(n) = \frac{M_n}{3} \]
	where $\Phi_r(n) = \Phi\left( \begin{array}{cc|r} 2 & -5 & n \\ 2 & +5 & r \end{array} \right)$. Moreover, the terms $\Phi_r(n)$ satisfy the following polynomial identity 
	\[(20)^{\lfloor{\frac{n-1}{2}}\rfloor}\frac{z^n-t^n}{z^2-t^2} = \] 
	\[  \sum_{r=0}^{\lfloor{\frac{n-1}{2}}\rfloor}
	\Phi_r(n) ( 2z^2+5zt +2t^2)^{\lfloor{\frac{n-1}{2}}\rfloor -r} (2z^2 -5zt+2t^{2})^{r}. \]};
\end{tikzpicture}

		\subsection{THE MERSENNE TRAJECTORY} 
Let $n$ be any odd natural number. The $\Phi$ trajectory of order $n$ from $(2,-5)$ to $(2,+5)$ starts with $M_n$ and ends with $\frac{2^n+1}{3}$, where $M_n$ are the Mersenne numbers defined by $M_n=2^n-1$. We call this specific trajectory the Mersenne trajectory at level $n$.  \\

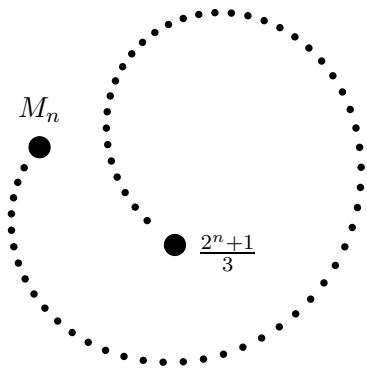
\makeatletter
\define@key{cylindricalkeys}{angle}{\def\myangle{#1}}
\define@key{cylindricalkeys}{radius}{\def\myradius{#1}}
\define@key{cylindricalkeys}{z}{\def\myz{#1}}
\tikzdeclarecoordinatesystem{cylindrical}%
{%
	\setkeys{cylindricalkeys}{#1}%
	\pgfpointadd{\pgfpointxyz{0}{0}{\myz}}{\pgfpointpolarxy{\myangle}{\myradius}}
}
\begin{tikzpicture}[z=0.1pt]
\node (a) at (-1,2.3) {$M_n$};
\node (b) at (1.5,0.4) {$\frac{2^n+1}{3}$};
\filldraw [] 
(0.8,0.5) circle (4pt);
\filldraw [] 
(-1,1.8) circle (4pt);
\foreach \num in {150,157,...,580}
\fill (cylindrical cs:angle=\num,radius=2,z=\num) circle (1.4pt);
\node (c) at (2.8,1.2)  [right=1cm,text width=9cm,font=\footnotesize] 
{For any odd natural number $n$, the Mersenne trajectory at level $n$ is the sequence
	\[M_n = \Phi_0(n), \Phi_1(n), \Phi_2(n), \dots, \Phi_{\lfloor{\frac{n}{2}}\rfloor}(n) = \frac{2^n+1}{3} \]
	where $\Phi_r(n) = \Phi\left( \begin{array}{cc|r} 2 & -5 & n \\ 2 & +5 & r \end{array} \right)$. Moreover, the terms $\Phi_r(n)$ satisfy the following polynomial identity 
	\[(20)^{\lfloor{\frac{n-1}{2}}\rfloor}\frac{z^n-t^n}{z-t} = \] 
	\[  \sum_{r=0}^{\lfloor{\frac{n-1}{2}}\rfloor}
	\Phi_r(n) ( 2z^2+5zt +2t^2)^{\lfloor{\frac{n-1}{2}}\rfloor -r} (2z^2 -5zt+2t^{2})^{r}. \]};
\end{tikzpicture}

 		\subsection{THE CHEBYSHEV-DICKSON TRAJECTORY OF THE FIRST KIND} 
 We can establish trajectory that connects the Chebyshev polynomials of first kind with
 the Dickson polynomials of the first kind with parameter $\alpha$. For example, the $\Psi$ trajectory of order $n$ from $(1,2-4x_1^2)$ to $(-\alpha,-2\alpha+x_2^2)$ starts with $\frac{2^{\delta(n+1)}}{x_1^{\delta(n)}}T_n(x_1)$ and ends with $x_2^{-\delta(n)}D_n(x_2, \alpha)$. So this way we get new path connects the Chebyshev polynomials of the first kind with the Dickson polynomials of the first kind with parameter $\alpha$. \\
 
 \makeatletter
 \define@key{cylindricalkeys}{angle}{\def\myangle{#1}}
 \define@key{cylindricalkeys}{radius}{\def\myradius{#1}}
 \define@key{cylindricalkeys}{z}{\def\myz{#1}}
 \tikzdeclarecoordinatesystem{cylindrical}%
 {%
 	\setkeys{cylindricalkeys}{#1}%
 	\pgfpointadd{\pgfpointxyz{0}{0}{\myz}}{\pgfpointpolarxy{\myangle}{\myradius}}
 }
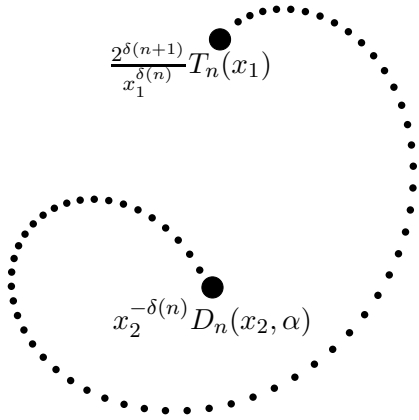
 \begin{tikzpicture}[z=0.2pt]
 \node (a) at (1.9,1.1) {$x_2^{-\delta(n)}D_n(x_2, \alpha)$};
 \node (b) at (1.6,4.4) {$\frac{2^{\delta(n+1)}}{x_1^{\delta(n)}}T_n(x_1)$};
 \filldraw [] 
 (2,4.8) circle (4pt);
 \filldraw [] 
 (1.9,1.5) circle (4pt);
 \foreach \num in {45,52,...,490}
 \fill (cylindrical cs:angle=\num,radius=2,z=\num) circle (1.4pt);
 \node (c) at (4.2,2.2)  [right=0.5cm,text width=9cm,font=\footnotesize] 
 {Studying this particular trajectory should help understand the arithmetic of Chebyshev polynomials of the first kind and Dickson polynomials of the first kind with parameter $\alpha$. We call this path the Chebyshev-Dickson trajectory of the first kind which is the sequence
 	\[ \frac{2^{\delta(n+1)}}{x_1^{\delta(n)}}T_n(x_1) = \Psi_0(n), \Psi_1(n),\dots, \Psi_{\lfloor{\frac{n}{2}}\rfloor}(n) = x_2^{-\delta(n)}D_n(x_2, \alpha)  \]
 	where $\Psi_r(n) = \Psi\left( \begin{array}{cc|r} 1 & 2-4x_1^2 & n \\ -\alpha & -2\alpha+x_2^2 & r \end{array} \right)$. Moreover, the terms $\Psi_r(n)$ satisfy the following polynomial identity 
 	\[(x_{2}^2 - 4\alpha x_{1}^2)^{\lfloor{\frac{n}{2}}\rfloor}\frac{z^n+t^n}{(z+t)^{\delta(n)}} = \] 
 	\[  \sum_{r=0}^{\lfloor{\frac{n}{2}}\rfloor}
 	\Psi_r(n) \Big( -\alpha z^2 -(2\alpha-x_2^2) zt  -\alpha t^2 \Big)^{\lfloor{\frac{n}{2}}\rfloor -r} (z^2+   (2-4x_1^2)zt+t^{2})^{r}. \] 
 };
 \end{tikzpicture}

  		\subsection{THE CHEBYSHEV-DICKSON TRAJECTORY OF THE SECOND KIND} 
 We can establish trajectory that connects the Chebyshev polynomials of second kind with
 the Dickson polynomials of the second kind with parameter $\alpha$. For example, the $\Phi$ trajectory of order $n$ from $(1,2-4x_1^2)$ to $(-\alpha,-2\alpha+x_2^2)$ starts with $(2x_1)^{-\delta(n-1)}U_{n-1}(x_1)$ and ends with $x_2^{-\delta(n-1)}E_{n-1}(x_2, \alpha)$. So this way we get new path connects Chebyshev polynomials of the second kind with Dickson polynomials of the second kind with parameter $\alpha$. \\
 
 \makeatletter
 \define@key{cylindricalkeys}{angle}{\def\myangle{#1}}
 \define@key{cylindricalkeys}{radius}{\def\myradius{#1}}
 \define@key{cylindricalkeys}{z}{\def\myz{#1}}
 \tikzdeclarecoordinatesystem{cylindrical}%
 {%
 	\setkeys{cylindricalkeys}{#1}%
 	\pgfpointadd{\pgfpointxyz{0}{0}{\myz}}{\pgfpointpolarxy{\myangle}{\myradius}}
 }
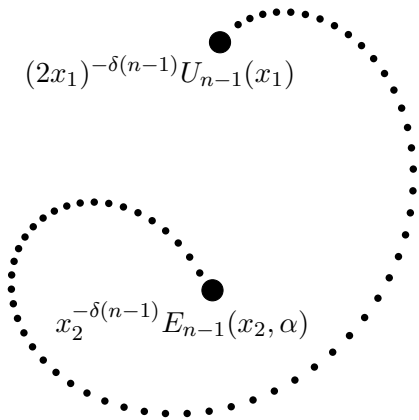
 \begin{tikzpicture}[z=0.2pt]
 \node (a) at (1.5,1.1) {$x_2^{-\delta(n-1)}E_{n-1}(x_2, \alpha)$};
 \node (b) at (1.2,4.4) {$(2x_1)^{-\delta(n-1)}U_{n-1}(x_1)$};
 \filldraw [] 
 (2,4.8) circle (4pt);
 \filldraw [] 
 (1.9,1.5) circle (4pt);
 \foreach \num in {45,52,...,490}
 \fill (cylindrical cs:angle=\num,radius=2,z=\num) circle (1.4pt);
 \node (c) at (4.2,2.2)  [right=0.5cm,text width=9cm,font=\footnotesize] 
 {This trajectory should help understand the arithmetic of the Chebyshev polynomials of the second kind and the Dickson polynomials of the second kind with parameter $\alpha$. We call this path the Chebyshev-Dickson trajectory of the second kind which is the sequence
 	\[ \frac{U_{n-1}(x_1)}{(2x_1)^{\delta(n-1)}} = \Phi_0(n), \Phi_1(n),\dots, \Phi_{\lfloor{\frac{n}{2}}\rfloor}(n) = \frac{E_{n-1}(x_2, \alpha)}{x_2^{\delta(n-1)}}  \]
 	where $\Phi_r(n) = \Phi\left( \begin{array}{cc|r} 1 & +2-4x_1^2 & n \\ -\alpha & -2\alpha+x_2^2 & r \end{array} \right)$. Moreover, the terms $\Phi_r(n)$ satisfy the following polynomial identity 
 	\[(x_{2}^2 - 4\alpha x_{1}^2)^{\lfloor{\frac{n-1}{2}}\rfloor}\frac{z^n-t^n}{(z-t)(z+t)^{\delta(n-1)}} = \] 
 	\[  \sum_{r=0}^{\lfloor{\frac{n-1}{2}}\rfloor}
 	\Phi_r(n) \Big( -\alpha z^2 -(2\alpha-x_2^2) zt  -\alpha t^2 \Big)^{\lfloor{\frac{n-1}{2}}\rfloor -r} (z^2+   (2-4x_1^2)zt+t^{2})^{r}. \] 
 };
 \end{tikzpicture}
 
 Therefore, as we see, with the $\Psi$ and $\Phi$ trajectories, many well-known sequences are not now isolated; meaning we can connect different elements from the list that contains Chebyshev polynomials of the first and second kind, Dickson polynomials of the first and second kind, Lucas and Fibonacci numbers, Mersenne numbers, Pell polynomials, Pell-Lucas polynomials, Fermat numbers, and $2^n \pm 1$. 
  \subsection*{Trajectory From Difference of Powers to Another and Open Questions} Also, for any natural number $n$, we show in this paper that we can also establish a trajectory that connects the polynomials $\frac{x^n-y^n}{(x-y) (x+y)^{\delta(n-1)}}$ with the polynomials $\frac{z^n-t^n}{(z-t) (z+t)^{\delta(n-1)}}$. \\
 
 \makeatletter
 \define@key{cylindricalkeys}{angle}{\def\myangle{#1}}
 \define@key{cylindricalkeys}{radius}{\def\myradius{#1}}
 \define@key{cylindricalkeys}{z}{\def\myz{#1}}
 \tikzdeclarecoordinatesystem{cylindrical}%
 {%
 	\setkeys{cylindricalkeys}{#1}%
 	\pgfpointadd{\pgfpointxyz{0}{0}{\myz}}{\pgfpointpolarxy{\myangle}{\myradius}}
 }
 \begin{tikzpicture}[z=0.2pt]
 \node (a) at (1.5,1.4) {$\frac{x^n-y^n}{(x-y) (x+y)^{\delta(n-1)}}$};
 \node (b) at (1.2,4.1) {$\frac{z^n-t^n}{(z-t) (z+t)^{\delta(n-1)}}$};
 \filldraw [] 
 (2.35,4.55) circle (4pt);
 \filldraw [] 
 (0.7,0.7) circle (4pt);
 \foreach \num in {240,249,...,490}
 \fill (cylindrical cs:angle=\num,radius=1.5,z=\num) circle (1.4pt);
 \node (c) at (4.2,2.3)  [right=0.15cm,text width=10cm,font=\footnotesize] 
 {There are many open questions to the trajectory of polynomials 
 	\[ \frac{x^n-y^n}{(x-y) (x+y)^{\delta(n-1)}} = \Phi_0(n), \Phi_1(n),\dots, \Phi_{\lfloor{\frac{n-1}{2}}\rfloor}(n) = \frac{z^n-t^n}{(z-t) (z+t)^{\delta(n-1)}} \]
 	where $\Phi_r(n) = \Phi\left( \begin{array}{cc|r} +xy & -x^2-y^2 & n \\ -zt & +z^2+t^2 & r \end{array} \right)$, which are the coefficients of the following polynomial expansion 
 	\[(zx-ty)^{\lfloor{\frac{n-1}{2}}\rfloor}(zy-tx)^{\lfloor{\frac{n-1}{2}}\rfloor}\frac{u^n-v^n}{(u-v)(u+v)^{\delta(n-1)}} = \] 
 	\[  \sum_{r=0}^{\lfloor{\frac{n-1}{2}}\rfloor}
 	\Phi_r(n) (zu-tv)^{\lfloor{\frac{n-1}{2}}\rfloor -r} (zv-tu)^{\lfloor{\frac{n-1}{2}}\rfloor -r} ((ux-vy)^{r}(uy-vx))^{r} \] 
 };
 \end{tikzpicture}
 
 For this particular trajectory, we observe that for any numbers $x,y,z,t,xz \neq yt, xt\neq yz$, and for any natural number $n$, the following identity hold 
 \[  \sum_{r=0}^{\lfloor{\frac{n-1}{2}}\rfloor} \Phi\left( \begin{array}{cc|r} +xy & -x^2-y^2 & n \\ -zt & +z^2+t^2 & r \end{array} \right) = \Phi(xy+zt,-x^2-y^2-z^2-t^2,n).
 \] 		
 
 This particular trajectory should also be surrounded by lots of mathematical questions. One of these, when the endpoints coincident with each other, again it gives an orbit. Also, if such orbit exists, we get a solution to many unsolved Diophantine equations. If such orbit exists, what are the general characteristics of such orbit? For example, what is the characteristics of this orbit, if exist, for $n=5$. And it is natural to ask whether there exist such orbit for any natural number $n$.  Alternatively, this is equivalent to ask whether there is  any nontrivial integer solutions for $(x,y,z,t)$ and any natural number $n$ for following Diophantine equation 
 \[     \frac{x^n-y^n}{(x-y)(x+y)^{\delta(n-1)}} = \frac{z^n-t^n}{(z-t)(z+t)^{\delta(n-1)}}. \]
	 \subsection*{Acknowledgments}
	 I would like to thank University of Bahrain for their kind support. 
	Special great thanks to Bruce Reznick for his kind motivation, great advice, and for useful conversations. I am also appreciated for Alexandru Buium, Bruce Berndt, and Ahmed Mater.

\medskip
\end{document}